\documentclass{amsart}

\usepackage{mathpazo, color}
\usepackage{amsmath, amscd}
\usepackage{amsthm}
\usepackage{amssymb}
\usepackage{amsfonts}
\usepackage{qsymbols}
\usepackage{latexsym}
\usepackage{cite}
\usepackage{graphicx}
\usepackage{url}
\usepackage{graphicx}
\newcommand {\R}{{\mathbb R}}

\newcommand {\C}{{\mathbb C}}

\newcommand {\N}{{{\mathbb N}}}

\newcommand {\Z}{{{\mathbb Z}}}

\newcommand{\A}{{\mathfrak{A}}}

\newcommand{\B}{{\mathcal{B}}}

\newcommand{\Br}{{\dot{\mathrm{B}}}}
\newcommand{\Ba}{{\mathfrak{B}}}
\newcommand{\ce}{\mathrm{c}}

\newcommand{\F}{{\mathcal{F}}}

\newcommand{\Ic}{{\mathcal{I}}}
\newcommand{\Ell}{\mathrm{L}}
\newcommand{\La}{{\mathcal{L}}}

\newcommand{\Pp}{{\mathcal{P}}}
\newcommand{\Qq}{{\mathcal{Q}}}

\newcommand{\Tr}{{\mathcal{T}}}
\newcommand{\W}{{\mathrm{W}}}
\newcommand{\ud}{\mathrm{d}}

\newcommand{\ph}{{\varphi}}
\newcommand{\w}{{\omega}}
\newcommand{\spA}{{\rm{sp}(A)}}

\newcommand{\specA}{{\nu(A)}}
\newcommand{\specB}{{\nu(B)}}
\newcommand{\funA}{{\nu(A)}}
\newcommand{\funB}{{\nu(B)}}
\newcommand{\diagA}{{K_{A}}}
\newcommand{\diagB}{{K_{B}}}

\newcommand{\scal}{{\La_{\textrm{s}}}}
\newcommand{\abs}[1]{\lvert#1\rvert}
\newcommand{\norm}[1]{\left\|#1\right\|}
\newcommand{\vanish}[1]{\relax}

\def\const{{\mathrm{const}\,}}

\newtheorem{theorem}{Theorem}[section]
\newtheorem{lemma}[theorem]{Lemma}
\newtheorem{problem}[theorem]{Problem}
\newtheorem{proposition}[theorem]{Proposition}
\newtheorem{corollary}[theorem]{Corollary}
\newtheorem{assumption}[theorem]{Assumption}

\theoremstyle{definition}
\newtheorem{definition}[theorem]{Definition}
\newtheorem{remark}[theorem]{Remark}

\numberwithin{equation}{section}

\title{Operator Lipschitz functions on Banach spaces}

\author{Jan Rozendaal}
\address{Delft Institute of Applied Mathematics\\Mekelweg 4\\2628CD Delft\\The Netherlands}
\email{janrozendaalmath@gmail.com}
\thanks{The first-named author is supported by NWO-grant
613.000.908 ``Applications of Transference Principles".}

\author{Fedor Sukochev}
\address{School of Mathematics \& Statistics, University of NSW,
  Kensington NSW 2052 AUSTRALIA}
\email{f.sukochev@unsw.edu.au}

\author{Anna Tomskova}
\address{}
\email{a.tomskova@unsw.edu.au}


\date{\today}

\dedicatory{}


\begin{document}

\begin{abstract}
Let $X$, $Y$ be Banach spaces and let $\La(X,Y)$ be the space of bounded linear operators from $X$ to $Y$. We develop the theory of double operator integrals on $\La(X,Y)$ and apply this theory to obtain commutator estimates of the form
\begin{align*}
\norm{f(B)S-Sf(A)}_{\La(X,Y)}\leq \const \norm{BS-SA}_{\La(X,Y)}
\end{align*}
for a large class of functions $f$, where $A\in\La(X)$, $B\in \La(Y)$ are scalar type operators and $S\in \La(X,Y)$. In particular, 
we establish this estimate for $f(t):=\abs{t}$ and for diagonalizable operators on $X=\ell_{p}$ and $Y=\ell_{q}$, for $p<q$ and $p=q=1$, and for $X=Y=\ce_{0}$. We also obtain results for $p\geq q$. 

We study the estimate above in the setting of Banach ideals in $\La(X,Y)$. The commutator estimates we derive hold for diagonalizable matrices with a constant independent of the size of the matrix.
\end{abstract}

\maketitle

\section{Introduction}
Let $X$ be a Banach space and let $\La(X)$ be the space of all bounded
linear operators on $X.$ Let $A,B \in\La(X)$ be scalar type
operators (see Definition \ref{scalar type operator} below) on
$X$. Let $f: {\rm sp}(A)\cup {\rm sp}(B)\to \mathbb C$ be a bounded Borel
function, where ${\rm sp}(A)$ (resp.~${\rm sp}(B)$) is the spectrum
of the operator $A$ (resp.~$B$). We are interested in
Lipschitz type estimates
\begin{align}\label{Lip_est_Intro}
\norm{f(B)-f(A)}_{\La(X)}\leq \const \norm{B-A}_{\La(X)},
\end{align}
where $\|\cdot\|_{\La(X)}$ is the uniform operator norm on the space $\La(X)$, and more generally in commutator estimates
\begin{align}\label{com_est_Intro}
\norm{f(B)S-Sf(A)}_{\La(X,Y)}\leq \const \norm{BS-SA}_{\La(X,Y)}
\end{align}
for Banach spaces $X$ and $Y$, scalar type
operators $A \in\La(X)$ and $B \in\La(Y)$, and $S\in\La(X,Y)$. This problem is well known in the special case where $X=Y$ is a separable Hilbert space, such as $\ell_{2}$, and $A$ and $B$ are normal operators on $X$. In this paper we study such estimates in the Banach space setting, and specifically for $X=\ell_{p}$ and $Y=\ell_{q}$ with $p,q\in[1,\infty]$.

In the special case where $A,B$ are self-adjoint bounded operators on a Hilbert space $H$ the estimate
\begin{align}\label{Lip_est_Intro_Hilbert}
\norm{f(B)-f(A)}_{\La(H)}\leq \const \norm{B-A}_{\La(H)}
\end{align}
was established by Peller \cite{Peller06, Peller85} (see also \cite{dPS}) for $f:\mathbb R\to\mathbb R$ in the Besov class $\Br^{1}_{\infty,1}(\R)$ (for the definition of $\Br^{1}_{\infty,1}(\R)$ see Section \ref{algebras of functions}).
This result extended a long line of results from \cite{BiSo1}-\cite{BiSo3}, in which the theory of double operator integration was developed to study the difference $f(B)-f(A)$ (see also \cite{Birman-Solomyak03}).
This theory was revised and extended in various directions, including the Banach space setting, in \cite{PWS02}.
However, until now the results in the general setting were much weaker than in the Hilbert space setting. 
In this paper we show that for scalar type operators on Banach spaces one can obtain results matching those on Hilbert spaces. 

In Corollary \ref{besov perturbation inequality1} below we prove that \eqref{Lip_est_Intro} holds when $A,B\in \La(X)$ are
scalar type operators with real spectrum and $f\in \Br^{1}_{\infty,1}(\R)$. It is immediate from the definition of a scalar type operator 
that every normal operator on $H$ is of scalar type. Therefore, Corollary \ref{besov perturbation inequality1} extends \eqref{Lip_est_Intro_Hilbert} to the Banach space setting. More generally, \eqref{com_est_Intro} holds for $f\in \Br^{1}_{\infty,1}(\R)$ and for all $S\in \La(X,Y)$ (see Corollary \ref{besov perturbation inequality}.)

If $f$ is the absolute value function then $f\notin\Br^{1}_{\infty,1}(\R)$ and the results mentioned above do not apply. Moreover, the techniques which we used to obtain
 \eqref{Lip_est_Intro} for $f\in\Br^{1}_{\infty,1}(\R)$ cannot be applied to the absolute value function (see Remark
 \ref{dimension-independent}).
However, the absolute value function is very important in the theory of matrix analysis and perturbation theory
(see~\cite[Sections VII.5 and X.2]{Bhatia97}).

In the case where $H$ is an infinite-dimensional Hilbert space, it was proved by Kato \cite{Kato} that the function
$t\mapsto|t|,$ $t\in\mathbb R$ does not satisfy \eqref{Lip_est_Intro_Hilbert}. An earlier example of  McIntosh \cite{McI} showed the failure in general of the commutator estimate \eqref{com_est_Intro} for this function, in the case $X=Y=H$. Later, it was
proved by Davies \cite{Davies} that for $1\le p\le\infty$ and the Schatten von-Neumann ideal $\mathcal S_p$ with the corresponding norm $\|\cdot\|_{\mathcal S_p}$, the estimate
\begin{align*}
\||B|-|A|\|_{\mathcal S_p}\le \const \|B-A\|_{\mathcal S_p}
\end{align*}
holds for all $A,B\in \mathcal S_p$ if and only if $1<p<\infty.$
Commutator estimates for the absolute value function and different Banach ideals in $\La(H)$ have also been studied
in \cite{DDPS}. The proofs in \cite{Davies, PWS02, DDPS} are based on Macaev's celebrated theorem (see \cite{GK}) or on the UMD-property of the reflexive Schatten von-Neumann ideals.
However, the spaces $\La(X,Y)$ are not UMD-spaces, and therefore the techniques used in \cite{Davies, PWS02, DDPS}  do not apply to them. To study \eqref{com_est_Intro} for $X=\ell_{p}$ and $Y=\ell_{q}$, we use completely different methods from those of \cite{Davies, PWS02, DDPS}. Instead, we use the theory of Schur multipliers on the space $\La(\ell_{p},\ell_{q})$ developed by Bennett \cite{Ben1},\cite{Ben}.

Let $p,q\in[1,\infty]$ with $p<q$, $p=q=1$ or $p=q=\infty$. In Section \ref{lp-spaces} we show (see Theorem \ref{truncation estimate_inf_lplq}) that, for diagonalizable operators (for the definition see Section \ref{unconditional basis spaces}) $A\in\La(\ell_{p})$ and $B\in\La(\ell_{q})$, and for the absolute value function $f$, 
\begin{align}\label{besov estimate_pq_Intro}
\norm{f(B)S-Sf(A)}_{\La(\ell_{p},\ell_{q})}\leq \const \,\norm{BS-SA}_{\La(\ell_{p},\ell_{q})}
\end{align}
holds for all $S\in \La(\ell_{p},\ell_{q})$ (where $\ell_{\infty}$ should be replaced by $\ce_{0}$). In particular,
\begin{align}\label{Lip_est_l1_Intro}
\norm{f(B)-f(A)}_{\La(\ell_{1})}\leq \const \norm{B-A}_{\La(\ell_{1})}
\end{align}
and
\begin{align}\label{Lip_est_c0_Intro}
\norm{f(B)-f(A)}_{\La(\ce_{0})}\leq \const \norm{B-A}_{\La(\ce_{0})}
\end{align}
for diagonalizable operators on $\ell_{1}$ respectively $\ce_{0}$. Therefore we show that, even though \eqref{besov estimate_pq_Intro} fails for $p=q=2$, and in particular \eqref{Lip_est_Intro} fails for $X=\ell_{2}$ and $f$ the absolute value function, one can obtain commutator estimates for $p<q$ and Lipschitz estimates \eqref{Lip_est_Intro} for $X=\ell_{1}$ or $X=\ce_{0}$. 

We also obtain results for $p\geq q$. In particular, for $p=q=2$ we prove (see Corollary \ref{case p=q=2}) that for each $\epsilon\in(0,1]$ there exists a constant $C\geq 0$ such that the following holds. Let $A,B\in\La(\ell_{2})$ be compact self-adjoint operators, and let $U,V\in\La(\ell_{2})$ be unitaries such that
\begin{align*}
UAU^{-1}=\sum_{j=1}^{\infty}\lambda_{j}\Pp_{j}\qquad\textrm{and}\qquad VBV^{-1}=\sum_{j=1}^{\infty}\mu_{j}\Pp_{j},
\end{align*}
where $\{\lambda_{j}\}_{j=1}^{\infty}$ and $\{\mu_{j}\}_{j=1}^{\infty}$ are sequences of real numbers and the $\Pp_{j}\in\La(\ell_{2})$, for $j\in\N$, are the basis projections corresponding to the standard basis of $\ell_{2}$. Then
\begin{align}\label{case p=q=2_Intro}
\norm{\abs{B}-\abs{A}}_{\La(\ell_{2})}\leq C\min\Big(\|V(B-A)U^{-1}\|_{\La(\ell_{2},\ell_{2-\epsilon})},\|V(B-A)U^{-1}\|_{\La(\ell_{2+\epsilon},\ell_{2})}\Big),
\end{align}
where we let the right-hand side equal infinity if $V(B-A)U^{-1}\notin\La(\ell_{2},\ell_{2-\epsilon})\cup\La(\ell_{2+\epsilon},\ell_{2})$.

We note that the constants which appear in our results depend on the spectral constants of $A$ and $B$ from Section \ref{scalar type operators}, and in \eqref{besov estimate_pq_Intro}, \eqref{Lip_est_l1_Intro} and \eqref{Lip_est_c0_Intro} on the diagonalizability constants of $A$ and $B$ from \eqref{definition diagonalizable constant}. These quantities are independent of the norms of $A$ and $B$, and to obtain constants which do not depend on $A$ and $B$ in any way one merely has to restrict to operators with a sufficiently bounded spectral or diagonalizability constant. This is already done implicitly on Hilbert spaces by considering normal operators, for which the quantities involved are equal to $1$. For example, in \eqref{case p=q=2_Intro} the constant $C$ does not depend on $A$ or $B$ in any way. Our results therefore truly extend the known estimates on Hilbert spaces, the main difference between Hilbert spaces and general Banach spaces being that on Hilbert spaces one has a large and easily identifiable class of operators with spectral constant $1$ or which are diagonalizable by an isometry.

We study the commutator estimate in \eqref{com_est_Intro} in the more general form
\begin{align}\label{com_est_ideal_Intro}
\norm{f(B)S-Sf(A)}_{\Ic}\leq \const \norm{BS-SA}_{\Ic},
\end{align}
where $\Ic$ is an operator ideal in $\La(X,Y)$. For example, we prove in Corollary \ref{besov perturbation inequality} that \eqref{Lip_est_Intro_Hilbert} holds for a general Banach ideal $\Ic$ in $\La(X)$ with the strong convex compactness property (for definitions see Section \ref{spaces of operators}), with respect to the norm $\norm{\cdot}_{\Ic}$. 

We also present (see Theorem \ref{th_p_summing_main}) an example of a Banach ideal $(\mathcal I,\|\cdot\|_{\mathcal I})$ in $\La(l_{p^*},\ell_{p})$, for $1<p<\infty$ and $\frac{1}{p}+\frac{1}{p^{*}}=1$, namely the ideal of $p$-summing operators, such that any Lipschitz function $f$ 
(in particular, the absolute value function) satisfies \eqref{com_est_ideal_Intro}.

In the final section we apply our results to finite-dimensional spaces, and obtain commutator estimates for diagonalizable matrices. Any diagonalizable matrix is a scalar type operator, hence estimates \eqref{besov estimate_pq_Intro}-\eqref{com_est_ideal_Intro} hold for diagonalizable matrices $A$ and $B$ with a constant independent of the size of the matrix.

\section{Notation and terminology}\label{notation}

All vector spaces are over the complex number field. Throughout, $X$ and $Y$ denote Banach spaces, the space of bounded linear operators from $X$ to $Y$ is $\La(X,Y)$, and $\La(X):=\La(X,X)$. We identify the algebraic tensor product $X^{*}\!\otimes Y$ with the space of finite rank operators in $\La(X,Y)$ via $(x^{*}\!\otimes y)(x):=\langle x^{*},x\rangle y$ for $x\in X$, $x^{*}\in X^{*}$ and $y\in Y$. The spectrum of $A\in\La(X)$ is ${\rm sp}(A)$, and by $\mathrm{I}_{X}\in\La(X)$ we denote the identity operator on $X$.
Throughout the text we use the abbreviations SOT and WOT for the strong and weak operator topology, respectively.

The Borel $\sigma$-algebra on a Borel measurable subset $\sigma\subseteq\C$ will be denoted by $\mathfrak{B}_{\sigma}$, and $\mathfrak{B}:=\mathfrak{B}_{\mathbb C}$. For measurable spaces $(\Omega_{1},\Sigma_{1})$ and $(\Omega_{2},\Sigma_{2})$ we denote by $\Sigma_{1}\otimes \Sigma_{2}$ the $\sigma$-algebra on $\Omega_{1}\times\Omega_{2}$ generated by all measurable rectangles $\sigma_1\times \sigma_2$ with $\sigma_1\in\Sigma_{1}$ and $\sigma_2\in\Sigma_{2}$. If $(\Omega,\Sigma)$ is a measurable space then $\B(\Omega,\Sigma)$ is the space of all bounded $\Sigma$-measurable complex-valued functions on $\Omega$, a Banach algebra with the supremum norm
\begin{align*}
\norm{f}_{\B(\Omega,\Sigma)}:=\sup_{\w\in\Omega}\,\abs{f(\omega)}\qquad\qquad(f\in\B(\Omega,\Sigma)).
\end{align*}
We simply write $\B(\Omega):=\B(\Omega,\Sigma)$ and $\norm{f}_{\infty}:=\norm{f}_{\B(\Omega,\Sigma)}$ when little confusion can arise.

If $\mu$ is a complex Borel measure on a measurable space $(\Omega,\Sigma)$ and $X$ is a Banach space, then a function $f:\Omega\rightarrow X$ is \emph{$\mu$-measurable} if there exists a sequence of $X$-valued simple functions converging to $f$ $\mu$-almost everywhere. For Banach spaces $X$ and $Y$ and a function $f:\Omega\rightarrow \La(X,Y)$, we say that $f$ is \emph{strongly measurable} if $\omega\mapsto f(\omega)x$ is a $\mu$-measurable mapping $\Omega\rightarrow Y$ for each $x\in X$.

{ If $\mu$ is a positive measure on a measurable space $(\Omega,\Sigma)$ and $f:\Omega\rightarrow[0,\infty]$ is a function, we let
\begin{align*}
\overline{\int_{\Omega}}f(\omega)\,\ud\mu(\omega):=\inf\int_{\Omega}g(\omega)\,\ud\mu(\omega)\in[0,\infty],
\end{align*}
where the infimum is taken over all measurable $g:\Omega\rightarrow [0,\infty]$ such that $g(\omega)\geq f(\omega)$ for $\omega\in\Omega$.}

The H\"{o}lder conjugate of $p\in[1,\infty]$ is denoted by $p^{*}$ and is defined by $\tfrac{1}{p}+\tfrac{1}{p^{*}}=1$. The indicator function of a subset $\sigma$ of a set $\Omega$ is denoted by $\mathbf{1}_{\sigma}$. We will often identify functions defined on $\sigma$ with their extensions to $\Omega$ by setting them equal to zero off $\sigma$.

\section{Preliminaries}\label{preliminaries}

\subsection{Scalar type operators}\label{scalar type operators}

In this section we summarize some of the basics of scalar type operators, as taken from \cite{Dunford-Schwartz1971}.

Let $X$ be a Banach space. A \emph{spectral measure} on $X$ is a map $E:\Ba\rightarrow \La(X)$ such that the following hold:
\begin{itemize}
\item $E(\emptyset)=0$ and $E(\C)=\mathrm{I}_{X}$;
\item $E(\sigma_1\cap \sigma_2)=E(\sigma_1)E(\sigma_2)$ for all $\sigma_1,\sigma_2\in\Ba$;
\item $E(\sigma_1\cup \sigma_2)=E(\sigma_1)+E(\sigma_2)-E(\sigma_1)E(\sigma_2)$ for all $\sigma_1,\sigma_2\in\Ba$;
\item $E$ is $\sigma$-additive in the strong operator topology.
\end{itemize}
Note that these conditions imply that $E$ is projection-valued. Moreover, by \cite[Corollary XV.2.4]{Dunford-Schwartz1971} there exists a constant $K$ such that
\begin{align}\label{spectral constant}
\norm{E(\sigma)}_{\La(X)}\leq K\qquad\qquad(\sigma\in\Ba).
\end{align}
An operator $A\in\La(X)$ is a \emph{spectral operator} if there exists a spectral measure $E$ on $X$ such that $A E(\sigma)=E(\sigma)A$ and ${\rm sp}(A,E(\sigma)X)\subseteq \overline{\sigma}$ for all $\sigma\in\Ba$, where ${\rm sp}(A,E(\sigma)X)$ denotes the spectrum of $A$ in the space $E(\sigma)X$. For a spectral operator $A$, we let $\specA$ denote the minimal constant $K$ occurring in \eqref{spectral constant} and call $\specA$ the \emph{spectral constant of $A$}. This is well-defined since the spectral measure $E$ associated with $A$ is unique, cf.~\cite[Corollary XV.3.8]{Dunford-Schwartz1971}. Moreover, $E$ is supported on ${\rm sp}(A)$ in the sense that $E({\rm sp}(A))=\mathrm{I}_{X}$ \cite[Corollary XV.3.5]{Dunford-Schwartz1971}.
Hence we can define an integral with respect to $E$ of bounded Borel measurable functions on ${\rm sp}(A)$, as follows. For $f=\sum_{j=1}^{n}\alpha_{j}\mathbf{1}_{\sigma_{j}}$ a finite simple function with $\alpha_{j}\in\C$ and $\sigma_{j}\subseteq {\rm sp}(A)$ Borel for $1\leq j\leq n$, we let
\begin{align}\label{simple functions}
\int_{{\rm sp}(A)}f\,\ud E:=\sum_{j=1}^{n}\alpha_{j}E(\sigma_{j}).
\end{align}
This definition is independent of the representation of $f$, and
\begin{align*}
\norm{\int_{{\rm sp}(A)}f\,\ud E}_{\La(X)}&=\sup_{\norm{x}_{X}=\norm{x^{*}}_{X^{*}}=1}\left|\sum_{j=1}^{n}\alpha_{j}x^{*}E(\sigma_{j})x\right|\\
&\leq \sup_{j}\,\abs{\alpha_{j}}\sup_{\norm{x}_{X}=\norm{x^{*}}_{X^{*}}=1}\norm{x^{*}E(\cdot)x}_{\textrm{var}} \\ &\leq 4 \norm{f}_{\B({\rm sp}(A))}\sup_{\norm{x}_{X}=\norm{x^{*}}_{X^{*}}=1}\sup_{\sigma\subseteq{\rm sp}(A)}\abs{x^{*}E(\sigma)x}\\
&\leq 4\specA \norm{f}_{\B({\rm sp}(A))},
\end{align*}
where $\norm{x^{*}E(\cdot)x}_{\textrm{var}}$ is the variation norm of the measure $x^{*}E(\cdot)x$. Since the simple functions lie dense in $\B({\rm sp}(A))$, for general $f\in \B({\rm sp}(A))$
we can define
\begin{align*}
\int_{{\rm sp}(A)}f\,\ud E:=\lim_{n\rightarrow\infty}\int_{{\rm sp}(A)}f_{n}\,\ud E\in \La(X)
\end{align*}
if $\{f_{n}\}_{n=1}^{\infty}\subseteq \B({\rm sp}(A))$ is a sequence of simple functions with $\norm{f_{n}-f}_{\infty}\rightarrow 0$ as $n\rightarrow \infty$. This definition is independent of the choice of approximating sequence and 
\begin{align}\label{norm bound for Borel functions}
\norm{\int_{{\rm sp}(A)}f\,\ud E}_{\La(X)}\leq 4\specA\norm{f}_{\B({\rm sp}(A))}.
\end{align}
It is straightforward to check that
\begin{align*}
\int_{{\rm sp}(A)}(\alpha f+g)\,\ud E&=\alpha\int_{{\rm sp}(A)}f\,\ud E+\int_{{\rm sp}(A)}g\,\ud E,\\
\int_{{\rm sp}(A)}fg\,\ud E&=\left(\int_{{\rm sp}(A)}f\,\ud E\right)\!\left(\int_{{\rm sp}(A)}g\,\ud E\right)
\end{align*}
for all $\alpha\in\C$ and simple $f,g\in\B({\rm sp}(A))$, and approximation extends these identities to general $f,g\in \B({\rm sp}(A))$. Moreover, $\int_{{\rm sp}(A)}\mathbf{1}\,\ud E=E({\rm sp}(A))=\mathrm{I}_{X}$. Hence the map $f\mapsto \int_{{\rm sp}(A)}f\,\ud E$ is a continuous morphism $\B({\rm sp}(A))\rightarrow \La(X)$ of unital Banach algebras. Since the spectrum of $A$ is compact, the identity function $\lambda\mapsto \lambda$ is bounded on ${\rm sp}(A)$ and $\int_{{\rm sp}(A)}\lambda\;\ud E(\lambda)\in\La(X)$ is well defined.

\begin{definition}\label{scalar type operator}
A spectral operator $A\in\La(X)$ with spectral measure $E$ is a \emph{scalar type operator} if
\begin{align*}
A=\int_{{\rm sp}(A)}\lambda\;\ud E(\lambda).
\end{align*}
The class of scalar type operators on $X$ is denoted by $\scal(X)$.
\end{definition}
For $A\in\scal(X)$ with spectral measure $E$ and $f\in\B({\rm sp}(A))$ we define
\begin{align}\label{functional calculus}
f(A):=\int_{{\rm sp}(A)}f\,\ud E.
\end{align}
Then, as remarked above, $f\mapsto f(A)$ is a continuous morphism $\B({\rm sp}(A))\rightarrow \La(X)$ of unital Banach algebras with norm bounded by $4\specA$. Note also that
\begin{align}\label{scalar measures}
\langle x^{*},f(A)x\rangle=\int_{{\rm sp}(A)}f(\lambda)\,\ud \langle x^{*},E(\lambda)x\rangle
\end{align}
for all $f\in\B({\rm sp}(A))$, $x\in X$ and $x^{*}\in X^{*}$. Indeed, for simple functions this follows from \eqref{simple functions}, and by taking limits one obtains \eqref{scalar measures} for general $f\in\B({\rm sp}(A))$.

Finally, we note that a normal operator $A$ on a Hilbert space $H$ is a scalar type operator with $\specA=1$, and in this case \eqref{norm bound for Borel functions} improves to
\begin{align}\label{norm bound for Borel functions and normal operators}
\norm{\int_{{\rm sp}(A)}f\,\ud E}_{\La(H)}\leq \norm{f}_{\B({\rm sp}(A))},
\end{align}
as known from the Borel functional calculus for normal operators.

\subsection{Spaces of operators}\label{spaces of operators}

In this section we discuss some properties of spaces of operators that we will need later on.

First we provide a lemma about approximation by finite rank operators. Recall that a Banach space $X$ has the \emph{bounded approximation property} if there exists $M\geq1$ such that, for each $K\subseteq X$ compact and $\epsilon>0$, there exists $S\in X^{*}\!\otimes\! X$ with $\norm{S}_{\La(X)}\leq M$ and $\sup_{x\in K}\norm{Sx-x}_{X}<\epsilon$.

\begin{lemma}\label{bounded approximation}
Let $X$ and $Y$ be Banach spaces such that $X$ is separable and either $X$ or $Y$ has the bounded approximation property. Then each $T\in\La(X,Y)$ is the SOT-limit of a norm bounded sequence of finite rank operators.
\end{lemma}
\begin{proof}
Fix $T\in\La(X,Y)$. By \cite[Proposition 1.e.14]{Lindenstrauss-Tzafriri1977} there exists a norm bounded net $\left\{T_{j}\right\}_{j\in J}\subseteq X^{*}\!\otimes Y$ having $T$ as its SOT-limit. It is straightfroward to see that the strong operator topology is metrizable on bounded subsets of $\La(X,Y)$ by
\begin{align*}
d(S_{1},S_{2}):=\sum_{k=1}^{\infty}2^{-k}\norm{S_{1}x_{k}-S_{2}x_{k}}_{Y}\qquad(S_{1},S_{2}\in\La(X,Y)),
\end{align*}
where $\left\{x_{k}\right\}_{k\in\N}\subseteq X$ is a countable subset that is dense in the unit ball of $X$. Hence there exists a subsequence of $\left\{T_{j}\right\}_{j\in J}$ with $T$ as its SOT-limit.
\end{proof}

Let $X$ and $Y$ be Banach spaces and let $Z$ be a Banach space which is continuously embedded in $\La(X,Y)$. Following \cite{Voigt92} (in the case where $Z$ is a subspace of $\La(X,Y)$), we say that $Z$ has the \emph{strong convex compactness property} if the following holds. For any finite measure space $(\Omega,\Sigma,\mu)$ and any strongly measurable bounded $f:\Omega\rightarrow Z$, the operator $T\in\La(X,Y)$ defined by
\begin{align}\label{sccp operator}
Tx:=\int_{\Omega}f(\omega)x\,\ud\mu(\omega)\qquad(x\in X),
\end{align}
belongs to $Z$ with $\norm{T}_{Z}\leq \overline{\int_{\Omega}}{\norm{f(\omega)}_{Z}}\,\ud\mu(\omega)$. By the Pettis Measurability Theorem, any separable $Z$ has this property. Indeed, if $Z$ is separable then combining Propositions 1.9 and 1.10 in \cite{VTC85} shows that any strongly measurable $f:\Omega\to Z$ is $\mu$-measurable as a map to $Z$. If $f$ is bounded as well, then \eqref{sccp operator} defines an element of $Z$ with 
\begin{align*}
\norm{T}_{Z}\leq \int_{\Omega}{\norm{f(\omega)}_{Z}}\,\ud\mu(\omega).
\end{align*}
It is shown in~\cite{Voigt92} and~\cite{Schluchtermann92} that the compact and weakly compact operators have the strong convex compactness property, but not all subspaces of $\La(X,Y)$ do. Moreover, if $\mathcal N$ is a semifinite von Neumann algebra on a separable Hilbert space $H$, with faithful normal semifinite trace $\tau$, and $\mathcal F$ is a rearrangement invariant Banach function space with the Fatou property, then $\mathcal E=\mathcal N\cap \mathcal F(N,\tau)$ has the strong convex compactness property~\cite[Lemma 3.5]{ACDS09}.

\begin{lemma}\label{sccp}
Let $X$ and $Y$ be separable Banach spaces and $Z$ a Banach space continuously embedded in $\La(X,Y)$. If $B_{Z}:=\left\{z\in Z\mid \norm{z}_{Z}\leq 1\right\}$ is SOT-closed in $\La(X,Y)$, then $Z$ has the strong convex compactness property.
\end{lemma}
\begin{proof}
The proof follows that of \cite[Lemma 3.5]{ACDS09}. First we show that $B_{Z}$ is a Polish space in the strong operator topology. As in the proof of Lemma \ref{bounded approximation}, bounded subsets of $\La(X,Y)$ are SOT-metrizable. The finite rank operators are SOT-dense in $\La(X,Y)$, hence $\La(X,Y)$ is SOT-separable. Therefore $B_{Z}$ is SOT-separable and metrizable. By assumption, $B_{Z}$ is complete.

Now let $(\Omega,\mu)$ be a finite measure space and let $f:\Omega\rightarrow Z$ be bounded and strongly measurable. Without loss of generality, we may assume that $f(\Omega)\subseteq B_{Z}$ and that $\mu$ is a probability measure. For each $y^{*}\in Y^{*}$ and $x\in X$, the mapping $B_{Z}\rightarrow [0,\infty)$, $T\mapsto \abs{\langle y^{*},Tx\rangle}$ is continuous. The collection of all these mappings, for $y^{*}\in Y^{*}$ and $x\in X$, separates the points of $B_{Z}$. Moreover, $\omega\mapsto \abs{\langle y^{*},f(\omega)x\rangle}$ is a measurable mapping $\Omega\rightarrow [0,\infty)$ for each $y^{*}\in Y^{*}$ and $x\in X$. By \cite[Propositions 1.9 and 1.10]{VTC85}, $f$ is the $\mu$-almost everywhere SOT-limit of a sequence of $B_{Z}$-valued simple functions $\{f_{k}\}_{k=1}^{\infty}$. Let $T_{n}:=\int_{\Omega}f_{n}\,\ud\mu\in B_{Z}$. By the dominated convergence theorem, $T_{n}(x)\rightarrow T(x):=\int_{\Omega}f(\omega)x\,\ud\mu(\omega)$ as $n\rightarrow \infty$, for all $x\in X$. By assumption, $T\in B_{Z}$.

Now let $g:\Omega\rightarrow [0,\infty)$ be measurable such that $1\geq g(\omega)\geq \norm{f(\omega)}_{Z}$ for $\omega\in\Omega$, and define $h(\omega):=\frac{f(\omega)}{g(\omega)}$ and $\ud\nu(\omega):=\frac{g(\omega)}{\int_{\Omega}g(\eta)\,\ud\mu(\eta)}\ud\mu(\omega)$ for $\omega\in\Omega$. By what we have shown above, $x\mapsto \int_{\Omega}h(\omega)x\,\ud\nu(\omega)$ defines an element of $B_{Z}$. Since $Tx=\int_{\Omega}f(\omega)x\,\ud\mu(\omega)=\int_{\Omega}g(\omega)\,\ud\mu(\omega)\int_{\Omega}h(\omega)x\,\ud\nu(\omega)$, we obtain $\norm{T}_{Z}\leq \int_{\Omega}g(\omega)\,\ud\mu(\omega)$, as remained to be shown.
\end{proof}

\begin{remark}\label{no equivalence}
Note that the converse implication does not hold. Indeed, if $X$ is a Hilbert space (or more generally, a Banach space with the metric approximation property) then the finite rank operators of norm less than or equal to $1$ are SOT-dense in the unit ball of $\La(X)$. Therefore the compact operators of norm less than or equal to $1$ are not SOT-closed in $\La(X)$ if $X$ is infinite-dimensional. However, by \cite[Theorem 1.3]{Voigt92}, the space of compact operators on $X$ has the strong convex compactness property.
\end{remark}

Let $X$ and $Y$ be Banach spaces and $\mathcal I$ a Banach space which is continuously embedded in $\La(X,Y)$. We say that
$(\Ic, \|\cdot\|_{\Ic})$ is a \emph{Banach ideal} in $\La(X,Y)$ if
\begin{itemize}
\item For all $R\in\La(Y)$, $S\in \Ic$ and $T\in\La(X)$, $RST\in \Ic$ with $\norm{RST}_{\Ic}\leq \norm{R}_{\La(Y)}\norm{S}_{\Ic}\norm{T}_{\La(X)}$;
\item $X^{*}\!\otimes Y\subseteq \Ic$ with $\norm{x^{*}\!\otimes y}_{\Ic}=\norm{x^{*}}_{X^{*}\!}\norm{y}_{Y}$ for all $x^{*}\in X^{*}$ and $y\in Y$.
\end{itemize}
By Lemma \ref{sccp} and \cite[Proposition 17.21]{Defant-Floret93} (using that the SOT and WOT closures of a convex set coincide), for separable $X$ and $Y$, any \emph{maximal} Banach ideal (for the definition see e.g. \cite{Pietsch}) in $\La(X,Y)$ has the strong convex compactness property. This includes a large class of operator ideals, such as the ideal of absolutely $p$-summing operators, the ideal of integral operators, etc.~\cite[p.~203]{Defant-Floret93}.

\subsection{Algebras of functions}\label{algebras of functions}

In this section we discuss some algebras of functions that will be essential in later sections.

Let $\sigma_1,\sigma_2\subseteq \C$ be Borel measurable subsets and let $\mathfrak{A}(\sigma_1\times \sigma_2)$ be the class of Borel functions $\ph:\sigma_1\times \sigma_2\rightarrow \C$ such that
\begin{align}\label{function representation}
\ph(\lambda_{1},\lambda_{2})=\int_{\Omega}a_{1}(\lambda_{1},\w)a_{2}(\lambda_{2},\w)\,\ud\mu(\w)
\end{align}
for all $(\lambda_{1},\lambda_{2})\in \sigma_1\times \sigma_2$, where $(\Omega,\Sigma,\mu)$ is a finite measure space and $a_{1}\in\B(\sigma_1\times\Omega,\mathfrak B_{\sigma_1}\otimes\Sigma)$, $a_{2}\in\B(\sigma_2\times\Omega,\mathfrak B_{\sigma_2}\otimes\Sigma)$. For $\ph\in\mathfrak{A}(\sigma_1\times \sigma_2)$ let
\begin{align*}
\norm{\ph}_{\mathfrak{A}(\sigma_1\times \sigma_2)}:=\inf\int_{\Omega}\norm{a_{1}(\cdot,\w)}_{\B(\sigma_1)}\norm{a_{2}(\cdot,\w)}_{\B(\sigma_2)}\ud\mu(\w),
\end{align*}
where the infimum runs over all possible representations\footnote{This might seem problematic from a set-theoretic viewpoint. The problem can be fixed by choosing an equivalence class of such a representation for each $r\in\R$ which can occur in the infimum.} in \eqref{function representation} (it is straightforward to show that the map $\omega\mapsto \norm{a_{1}(\cdot,\w)}_{\B(\sigma_1)}\norm{a_{2}(\cdot,\w)}_{\B(\sigma_2)}$ is measurable).

\begin{lemma}\label{Banach algebra}
For all $\sigma_1,\sigma_2\subseteq \C$ measurable, $\mathfrak{A}(\sigma_1\times \sigma_2)$ is a unital Banach algebra which is contractively included in $\B(\sigma_1\times \sigma_2)$.
\end{lemma}
\begin{proof}
That $\mathfrak{A}(\sigma_1\times \sigma_2)$ is a vector space is straightforward, and that it is normed algebra is shown in \cite[Lemma 3]{Potapov-Sukochev2010} for $\sigma_1=\sigma_2=\R$ (the proof in our setting is identical). The completeness of $\mathfrak{A}(\sigma_1\times \sigma_2)$ follows by showing that an absolutely convergent series of elements in $\mathfrak{A}(\sigma_1\times \sigma_2)$ converges in $\mathfrak{A}(\sigma_1\times \sigma_2)$. This is done by considering a direct sum of the measure spaces involved.
\end{proof}

We now state sufficient conditions for a function to belong to $\mathfrak{A}$. The first will be used in the proof of Proposition \ref{truncation estimate_inf}. Let $\W^{1,2}(\R)$ be the space of all $g\in \Ell^{2}(\R)$ with weak derivative $g'\in\Ell^{2}(\R)$, endowed with the norm $\norm{g}_{\W^{1,2}(\R)}:=\norm{g}_{\Ell^{2}(\R)}+\norm{g'}_{\Ell^{2}(\R)}$ for $g\in\W^{1,2}(\R)$.

\begin{lemma}\cite[Theorem 9]{Potapov-Sukochev2010}\label{sobolev condition}
Let $g\in\W^{1,2}(\R)$ and let
\begin{align}\label{psi_f}
\psi_g(\lambda_{1},\lambda_{2}):= \left\{\begin{array}{ll}
g(\log(\frac{\lambda_{1}}{\lambda_{2}}))&\textrm{if }\lambda_{1},\lambda_{2}>0\\
0&\textrm{otherwise}
\end{array}.\right.
\end{align}
Then $\psi_g\in\mathfrak{A}(\R^{2})$ with $\big\|\psi_g\big\|_{\mathfrak{A}(\R^{2})}\leq
\sqrt{2}\norm{g}_{\W^{1,2}(\R)}$.
\end{lemma}

The second condition involves the Besov space $\Br^{1}_{\infty,1}(\R)$. Following \cite{Peller1990}, let $\{\psi_{k}\}_{k\in\Z}$ be a sequence of Schwartz functions on $\R$ such that, for each $k\in\Z$, the Fourier transform $\F\psi_{k}$ of $\psi_{k}$ is supported on $[2^{k-1},2^{k+1}]$ and $\F\psi_{k+1}(x)=\F\psi_{k}(2x)$ for all $x>0$, and such that $\sum_{k=-\infty}^{\infty}\F\psi_{k}(x)=1$ for all $x>0$. Let $\psi_{k}^{*}$ be defined by $\F\psi_{k}^{*}(x)=\F\psi_{k}(-x)$ for $k\in\Z$ and $x\in\R$. If $f$ is a distribution on $\R$ such that $\{2^{k}\norm{f\ast \psi_{k}}_{\Ell^{\infty}\!(\R)}\}_{k\in\Z}\in \ell_1(\Z)$ and $\{2^{k}\norm{f\ast \psi^{*}_{k}}_{\Ell^{\infty}\!(\R)}\}_{k\in\Z}\in \ell_1(\Z)$, then $f$ admits a representation
\begin{align}\label{besov representation}
f=\sum_{k\in\Z}f\ast\psi_{k}+\sum_{k\in\Z}f\ast\psi^{*}_{k}+P,
\end{align}
where $P$ is a polynomial.

We let the \emph{homogeneous Besov space} $\Br^{1}_{\infty,1}(\R)$ consist of all distributions as above for which $P=0$. Then $\Br^{1}_{\infty,1}(\R)$ is a Banach space when equipped with the norm
\begin{align*}
\norm{f}_{\Br^{1}_{\infty,1}(\R)}:=\sum_{k=-\infty}^{\infty}2^{k}\norm{f\ast \psi_{k}}_{\Ell^{\infty}\!(\R)}+\sum_{k=-\infty}^{\infty}2^{k}\norm{f\ast \psi^{*}_{k}}_{\Ell^{\infty}\!(\R)}\qquad(f\in \Br^{1}_{\infty,1}(\R)).
\end{align*}
For $f\in\Br^{1}_{\infty,1}(\R)$ define $\psi_{f}:\C^{2}\rightarrow \C$ by

\begin{align*}
\psi_{f}(\lambda_{1},\lambda_{2}):=
\left\{\begin{array}{ll}
\frac{f(\lambda_{2})-f(\lambda_{1})}{\lambda_{2}-\lambda_{1}}&\textrm{if }(\lambda_{1},\lambda_{2})\in\R^{2}\textrm{ and }\lambda_{1}\neq\lambda_{2}\\
f'(\lambda_{1})&\textrm{if }\lambda_{1}=\lambda_{2}\in\R
\end{array}.\right.
\end{align*}

\begin{lemma}\label{besov condition}
There exists a constant $C\geq 0$ such that $\psi_{f}\in\mathfrak{A}(\R^{2})$ for each $f\in\Br^{1}_{\infty,1}(\R)$, with $\norm{\psi_{f}}_{\mathfrak{A}(\R^{2})}\leq C\norm{f}_{\Br^{1}_{\infty,1}(\R)}$.
\end{lemma}
\begin{proof}
Let $f\in\Br^{1}_{\infty,1}(\R)$. In \cite[Theorem 2]{Peller1990} (see also \cite[p.~535]{Peller06}) it is shown that $\psi_{f}$ has a representation
\begin{align*}
\psi_{f}(\lambda_{1},\lambda_{2})=\int_{\Omega}a_{1}(\lambda_{1},\omega)a_{2}(\lambda_{2},\omega)\,\ud\mu(\w)
\end{align*}
for $(\lambda_{1},\lambda_{2})\in\R^{2}$, where $(\Omega,\mu)$ is a measure space and $a_{1}$ and $a_{2}$ are measurable functions on $\R\times\Omega$ such that
\begin{align*}
\int_{\Omega}\norm{a_{1}(\cdot,\w)}_{\infty}\norm{a_{2}(\cdot,\w)}_{\infty}\ud\abs{\mu}(\w)\leq C\norm{f}_{\Br^{1}_{\infty,1}(\R)}
\end{align*}
for some constant $C\geq 0$ independent of $f$. The desired conclusion now follows by replacing $a_{1}(\lambda_{1},\w)$ and $a_{2}(\lambda_{2},\w)$ by $\frac{a_{1}(\lambda_{1},\w)}{\norm{a_{1}(\cdot,\w)}_{\infty}}$ respectively $\frac{a_{2}(\lambda_{2},\w)}{\norm{a_{2}(\cdot,\w)}_{\infty}}$, and $\ud\mu(\w)$ by $\norm{a_{1}(\cdot,\w)}_{\infty}\norm{a_{2}(\cdot,\w)}_{\infty}\ud\mu(\omega)$.
\end{proof}

In \cite[Theorem 3]{Peller1990} Peller also states a condition on a function $f$ on $\R$ which is necessary in order for $\ph_{f}\in\mathfrak{A}(\R^{2})$ to hold, and this condition is only slightly weaker than $f\in \Br^{1}_{\infty,1}(\R)$.

\section{Double operator integrals and Lipschitz estimates}\label{double operator integrals and Lipschitz estimates}

\subsection{Double operator integrals}\label{double operator integrals}

Fix Banach spaces $X$ and $Y$, scalar type operators $A\in\scal(X)$ and $B\in\scal(Y)$ with spectral measures $E$ respectively $F$, and $\ph\in\A({\rm sp}(A)\times{\rm sp}(B))$. Let a representation \eqref{function representation} for $\ph$ be given, with corresponding $(\Omega,\mu)$ and $a_{1}\in\B({\rm sp}(A)\times\Omega)$, $a_{2}\in\B({\rm sp}(B)\times\Omega)$. For $\omega\in\Omega$, let $a_{1}(A,\w):=a_{1}(\cdot,\w)(A)\in\La(X)$ and $a_{2}(B,\w):=a_{2}(\cdot,\w)(B)\in\La(Y)$ be defined by the functional calculus for $A$ respectively $B$ from Section \ref{scalar type operators}.

\begin{lemma}\label{measurability}
Let $S\in \La(X,Y)$ have separable range. Then, for each $x\in X$, $\w\mapsto a_{2}(B,\w)S a_{1}(A,\w)x$ is a weakly measurable map $\Omega\rightarrow Y$.
\end{lemma}
\begin{proof}
Fix $x\in X$. If $a_{1}=\mathbf{1}_{\sigma}$ for some $\sigma\subseteq {\rm sp}(A)\times \Omega$ then it is straightforward to show that $\langle x^{*},a_{1}(A,\cdot)x\rangle$ is measurable for each $x^{*}\in X^{*}$. As $S$ has separable range, $Sa_{1}(A,\cdot)x$ is $\mu$-measurable. If $a_{2}$ is an indicator function as well, the same argument shows that $a_{2}(B,\cdot)y$ is weakly measurable for each $y\in Y$. General arguments, approximating $Sa_{1}(A,\cdot)x$ by simple functions, show that $a_{2}(B,\cdot)Sa_{1}(A,\cdot)x$ is weakly measurable. By linearity this extends to simple $a_{1}$ and $a_{2}$, and for general $a_{1}$ and $a_{2}$ let $\{f_{k}\}_{k\in\N}$, $\{g_{k}\}_{k\in\N}$ be sequences of simple functions such that $a_{1}=\lim_{k\rightarrow\infty}f_{k}$ and $a_{2}=\lim_{k\rightarrow\infty}g_{k}$ uniformly. Then $a_{1}(A,\w)=\lim_{k\rightarrow\infty}f_{k}(A)$ and $a_{2}(B,\w)=\lim_{k\rightarrow\infty}g_{k}(B)$ in the operator norm, for each $\w\in\Omega$. The desired measurability now follows.
\end{proof}

Now suppose that $Y$ is separable, that $\Ic$ is a Banach ideal in $\La(X,Y)$ and let $S\in\La(X,Y)$. By \eqref{norm bound for Borel functions},
\begin{align}\label{operator norm inequality}
\norm{a_{2}(B,\w)S a_{1}(A,\w)}_{\Ic}\leq 16\,\funA\funB\norm{S}_{\Ic}\norm{a_{1}(\cdot,\w)}_{\B({\rm sp}(A))}\norm{a_{2}(\cdot,\w)}_{\B({\rm sp}(B))}
\end{align}
for $w\in\Omega$. Since $\mathcal{I}$ is continuously embedded in $\La(X,Y)$, by the Pettis Measurability Theorem, Lemma \ref{measurability} and \eqref{operator norm inequality} we can define the \emph{double operator integral}
\begin{align}\label{double operator integral}
T^{A,B}_{\ph}(S)x:=\int_{\Omega}
a_{2}(B,\w)S a_{1}(A,\w)x\,\ud\mu(\w)\in Y\qquad(x\in X).
\end{align}
Throughout, we will use $T_{\ph}$ for $T_{\ph}^{A,B}$ when the operators $A$ and $B$ are clear from the context.

\begin{proposition}\label{operator integral well defined}
Let $X$ and $Y$ be separable Banach spaces such that $X$ or $Y$ has the bounded approximation property, and let $A\in \La_{s}(X)$, $B\in\La_{s}(Y)$, and $\ph\in\A({\rm sp}(A)\times{\rm sp}(B))$. Let $\Ic$ be a Banach ideal in $\La(X,Y)$ with the strong convex compactness property. Then \eqref{double operator integral} defines an operator $T^{A,B}_{\ph}\in\La(\Ic)$ which is independent of the choice of representation of
$\ph$ in \eqref{function representation}, with 
\begin{align}\label{double operator estimate}
\norm{T^{A,B}_{\ph}}_{\La(\Ic)}\leq 16\,\funA\funB\norm{\ph}_{\A({\rm sp}(A)\times{\rm sp}(B))}.
\end{align}
\end{proposition}
\begin{proof}
By \eqref{operator norm inequality} and the strong convex compactness property, $T_{\ph}(S)\in\La(\Ic)$ for all $S\in \Ic$, with
\begin{align*}
\norm{T_{\ph}(S)}_{\Ic}\leq 16\,\funA\funB\norm{S}_{\Ic}\int_{\Omega}\norm{a_{1}(\cdot,\omega)}_{\B({\rm sp}(A))}\norm{a_{2}(\cdot,\omega)}_{\B({\rm sp}(B))}\ud\mu(\omega).
\end{align*}
Clearly $T_{\ph}$ is linear, hence the result follows if we establish that $T_{\ph}$ is independent of the representation of $\ph$. For this it suffices to let $\ph\equiv 0$. Now, first consider $S=x^{*}\!\otimes y$ for $x^{*}\in X^{*}$ and $y\in Y$, and let $x\in X$, $y^{*}\in Y^*$ and $w\in\Omega$.
Recall that $E$ and $F$ are the spectral measures of $A$ and $B$, respectively.
Then 
\begin{align*}
\langle y^{*},a_{2}(B,\w)S a_{1}(A,\w)x\rangle&=\int_{{\rm sp}(B)}a_{2}(\eta,\w)\,\ud \langle y^{*},F(\eta)Sa_{1}(A,\w)x\rangle\\
&=\int_{{\rm sp}(B)}a_{2}(\eta,\w)\langle x^{*},a_{1}(A,\w)x\rangle\,\ud \langle y^{*},F(\eta)y\rangle\\
&=\int_{{\rm sp}(B)}\int_{{\rm sp}(A)}a_{1}(\lambda,\w)a_{2}(\eta,\w)\ud \langle x^{*},E(\lambda)x\rangle\ud \langle y^{*},F(\eta)y\rangle
\end{align*}
by \eqref{scalar measures}. Now Fubini's Theorem and the assumption on $\ph$ yield $T_{\ph}(S)=0$. By linearity, $T_{\ph}(S)=0$ for all $S\in X^{*}\otimes Y$. By Lemma \ref{bounded approximation}, a general $S\in \Ic$ is the SOT-limit of a bounded (in $\La(X,Y)$) sequence $\{S_{n}\}_{n\in\N}\subseteq X^{*}\otimes Y$. The dominated convergence theorem shows that $T_{\ph}(S)x=\lim_{n\rightarrow\infty}T_{\ph}(S_{n})x=0$ for all $x\in X$, which implies that $T_{\ph}$ is independent of the representation of $\ph$ and concludes the proof.
\end{proof}

Note that, if $A$ and $B$ are normal operators on separable Hilbert spaces $X$ and $Y$, then \eqref{double operator estimate} improves to
\begin{align}\label{double operator estimate normal}
\norm{T^{A,B}_{\ph}}_{\La(\Ic)}\leq \norm{\ph}_{\A({\rm sp}(A)\times{\rm sp}(B))},
\end{align}
by appealing to \eqref{norm bound for Borel functions and normal operators} instead of \eqref{norm bound for Borel functions} in \eqref{operator norm inequality}. 

\begin{remark}\label{sharpness Peller}
Let $H$ be an infinite-dimensional separable Hilbert space and $\mathcal{S}_{2}$ the ideal of Hilbert-Schmidt operators on $H$. There is a natural definition (see~\cite{Birman-Solomyak03}) of a double operator integral $\mathcal{T}^{A,B}_{\ph}\in \La(\mathcal{S}_{2})$ for all $\ph\in\B(\C^{2})$ and normal operators $A,B\in\La(H)$, such that $\mathcal{T}^{A,B}_{\ph}=T^{A,B}_{\ph}$ if $\ph\in A({\rm sp}(A)\times{\rm sp}(B))$. One could wonder whether Proposition \ref{operator integral well defined} can be extended to a larger class of functions than $\A({\rm sp}(A)\times{\rm sp}(B))$ using an extension of the definition of $T_{\ph}^{A,B}$ in \eqref{double operator integral} which coincides with $\mathcal{T}^{A,B}_{\ph}$ on $\mathcal{S}_{2}$. But it follows from~\cite[Theorem 1]{Peller85} that $\mathcal{T}^{A,B}_{\ph}$ extends to a bounded operator on $\Ic=\La(H)$ if and only if $\ph\in\A({\rm sp}(A)\times{\rm sp}(B))$. Hence Proposition \ref{operator integral well defined} cannot be extended to a larger function class than $\A({\rm sp}(A)\times{\rm sp}(B))$ in general. However, for specific Banach ideals, e.g.~ideals with the UMD property, results have been obtained for larger classes of functions~\cite{PWS02, Potapov-Sukochev11}.
\end{remark}

\begin{remark}\label{no bap}
The assumption in Proposition \ref{operator integral well defined} that $X$ or $Y$ has the bounded approximation property is only used, via Lemma \ref{bounded approximation}, to ensure that each $S\in I$ is the SOT-limit of a bounded (in $\La(X,Y)$) net of finite-rank operators. Clearly this is true for general Banach spaces $X$ and $Y$ if $I$ is the closure in $\La(X,Y)$ of $X^{*}\!\otimes Y$. In \cite{Lassalle-Turco} the authors consider an assumption on $X$ and $I$, called condition $c_{\lambda}^{*}$, which guarantees that each $S\in I$ is the SOT-limit of a bounded net of finite-rank operators. It is shown in \cite{Lassalle-Turco} that this condition is strictly weaker than the bounded approximation property, for certain non-trivial ideals. In the results throughout the paper where we assume that $X$ has the bounded approximation property, one may assume instead that $X$ has satisfies condition $c_{\lambda}^{*}$ for $I$ for some $\lambda\geq 1$.

\end{remark}

\subsection{Commutator and Lipschitz estimates}\label{commutator estimates}

Let $p_{1}:,p_{2}:\C^{2}\rightarrow \C$ be the coordinate projections $p_{1}(\lambda_{1},\lambda_{2}):=\lambda_{1}$, $p_{2}(\lambda_{1},\lambda_{2}):=\lambda_{2}$ for $(\lambda_{1},\lambda_{2})\in \C^{2}$. Note that $f\circ p_{1}, f\circ p_{2}\in\A(\sigma_1\times \sigma_2)$ for all $\sigma_1,\sigma_2\subseteq \C$ Borel and $f\in\B(\sigma_1\cup \sigma_2)$. For $A$ and $B$ selfadjoint operators on a Hilbert space and $\Ic$ a non-commutative $\Ell_{p}$-space, the following lemma is \cite[Lemma 3]{Potapov-Sukochev2010}.

\begin{lemma}\label{functional calculus lemma}
Under the assumptions of Proposition \ref{operator integral well defined}, the following hold:
\begin{enumerate}
\item\label{multiplicativity} The map $\ph\mapsto T_{\ph}^{A,B}$ is a morphism $\A({\rm sp}(A)\times{\rm sp}(B))\rightarrow \La(\Ic)$ of unital Banach algebras.
\item\label{consistency} Let $f\in \B({\rm sp}(A)\cup{\rm sp}(B))$ and $S\in \La(X,Y)$. Then $T_{f\circ p_{1}}(S)=Sf(A)$ and $T_{f\circ p_{2}}(S)=f(B)S$. In particular, $T_{p_{1}}(S)=SA$ and $T_{p_{2}}(S)=BS$.
\end{enumerate}
\end{lemma}
\begin{proof}
The structure of the proof is the same as that of \cite[Lemma 3]{Potapov-Sukochev2010}. Linearity in \eqref{multiplicativity} is straightforward. Fix $\ph_{1},\ph_{2}\in \A({\rm sp}(A)\times{\rm sp}(B))$ with representations as in \eqref{function representation}, with corresponding measure spaces $(\Omega_{j},\mu_{j})$ and bounded Borel functions $a_{1,j}\in\B({\rm sp}(A)\times\Omega_{j})$ and $a_{2,j}\in\B({\rm sp}(B)\times\Omega_{j})$ for $j\in\left\{1,2\right\}$. Then $\ph:=\ph_{1}\ph_{2}$ also has a representation as in \eqref{function representation}, with $\Omega=\Omega_{1}\times\Omega_{2}$, $\mu=\mu_{1}\times\mu_{2}$ the product measure and $a_{1}=a_{1,1}a_{1,2}$, $a_{2}=a_{2,1}a_{2,2}$. By multiplicativity of the functional calculus for $A$,
\begin{align*}
a_{1}(A,(\omega_{1},\omega_{2}))=\big(a_{1,1}(\cdot,\omega_{1})a_{1,2}(\cdot,\omega_{2})\big)(A)=a_{1,1}(A,\omega_{1})a_{1,2}(A,\omega_{2})
\end{align*}
for all $(\omega_{1},\omega_{2})\in\Omega$, and similarly for $a_{2}(B,(\omega_{1},\omega_{2}))$. Applying this to \eqref{double operator integral} yields
\begin{align*}
T_{\ph}(S)x&=\int_{\Omega}
a_{2}(B,\w)S a_{1}(A,\w)x\,\ud\mu(\w)\\
&=\int_{\Omega_{1}}a_{2,1}(B,\w_{1})T_{\ph_{2}}(S)a_{1,1}(A,\w_{1})x\,\ud\mu_{1}(\omega_{1})\\
&=T_{\ph_{1}}(T_{\ph_{2}}(S))x
\end{align*}
for all $S\in \Ic$ and $x\in X$, which proves \eqref{multiplicativity}. Part \eqref{consistency} follows from \eqref{double operator integral} and the fact that $T_{\ph}$ is independent of the representation of $\ph$.
\end{proof}

For $f:{\rm sp}(A)\cup{\rm sp}(B)\rightarrow \C$ define
\begin{align}\label{divided difference}
\ph_{f}(\lambda_{1},\lambda_{2}):=\frac{f(\lambda_{2})-f(\lambda_{1})}{\lambda_{2}-\lambda_{1}}
\end{align}
for $(\lambda_{1},\lambda_{2})\in{\rm sp}(A)\times{\rm sp}(B)$ with $\lambda_{1}\neq\lambda_{2}$.

\begin{theorem}\label{general perturbation inequality}
Let $X$ and $Y$ be separable Banach spaces such that $X$ or $Y$ has the bounded approximation property, and let $\Ic$ be a Banach ideal in $\La(X,Y)$ with the strong convex compactness property. Let $A\in\La_{s}(X)$ and $B\in\La_{s}(Y)$, and let $f\in\B({\rm sp}(A)\cup{\rm sp}(B))$ be such that $\ph_{f}$ extends to an element of $\A({\rm sp}(A)\times{\rm sp}(B))$. Then
\begin{align}\label{main inequality}
\norm{f(B)S-Sf(A)}_{\Ic}\leq 16\,\funA\funB\big\|\ph_{f}\big\|_{\A({\rm sp}(A)\times{\rm sp}(B))}\norm{BS-SA}_{\Ic}
\end{align}
for all $S\in \La(X,Y)$ such that $BS-SA\in \Ic$.

In particular, if $X=Y$ and $B-A\in \Ic$,
\begin{align*}
\norm{f(B)-f(A)}_{\Ic}\leq 16\,\funA\funB\big\|\ph_{f}\big\|_{\A({\rm sp}(A)\times{\rm sp}(B))}\norm{B-A}_{\Ic}.
\end{align*}
\end{theorem}
\begin{proof}
Note that $(p_{2}-p_{1})\ph_{f}=f\circ p_{2}-f\circ p_{1}$. By Lemma \ref{functional calculus lemma},
\begin{align*}
f(B)S-Sf(A)&=T_{f\circ p_{2}}(S)-T_{f\circ p_{1}}(S)=T_{(p_{2}-p_{1})\ph_{f}}(S)=T_{p_{2}\ph_{f}}(S)-T_{p_{1}\ph_{f}}(S)\\
&=T_{\ph_{f}}(T_{p_{2}}(S)-T_{p_{1}}(S))=T_{\ph_{f}}(BS-SA)
\end{align*}
for each $S\in \Ic$. Proposition \ref{operator integral well defined} now concludes the proof.
\end{proof}

Letting $X$ and $Y$ be Hilbert spaces and $A$ and $B$ normal operators, we generalize results from~\cite{Birman-Solomyak03, Potapov-Sukochev2010} to all Banach ideals with the strong convex compactness property. As mentioned in Section \ref{spaces of operators}, this includes all separable ideals and the so-called maximal operator ideals, which in turn is a large class of ideals containing the absolutely $(p,q)$-summing operators, the integral operators, and more~\cite[p.~203]{Defant-Floret93}. Note that, for normal operators, we can improve the estimate in \eqref{main inequality} by appealing to \eqref{double operator estimate normal} instead of \eqref{double operator estimate}.

\begin{corollary}\label{normal operators}
Let $A\in\La(X)$ and $B\in\La(Y)$ be normal operators on separable Hilbert spaces $X$ and $Y$. Let $\Ic$ be a Banach ideal in $\La(X,Y)$ with the strong convex compactness property, and let $f\in\B({\rm sp}(A)\cup{\rm sp}(B))$ be such that $\ph_{f}$ extends to an element of $\A({\rm sp}(A)\times{\rm sp}(B))$. Then
\begin{align*}
\norm{f(B)S-Sf(A)}_{\Ic}\leq \big\|\ph_{f}\big\|_{\A({\rm sp}(A)\times{\rm sp}(B))}\norm{BS-SA}_{\Ic}
\end{align*}
for all $S\in \La(X,Y)$ such that $BS-SA\in \Ic$. In particular, if $X=Y$ and $B-A\in \Ic$,
\begin{align*}
\norm{f(B)-f(A)}_{\Ic}\leq \big\|\ph_{f}\big\|_{\A({\rm sp}(A)\times{\rm sp}(B))}\norm{B-A}_{\Ic}.
\end{align*}
\end{corollary}

Combining Theorem \ref{general perturbation inequality} with Lemma \ref{besov condition} yields the following, a generalization of~\cite[Theorem 4]{Peller1990}.

\begin{corollary}\label{besov perturbation inequality}
There exists a universal constant $C\geq 0$ such that the following holds. Let $X$ and $Y$ be separable Banach spaces such that $X$ or $Y$ has the bounded approximation property, and let $\Ic$ be a Banach ideal in $\La(X,Y)$ with the strong convex compactness property. Let $f\in\Br^{1}_{\infty,1}(\R)$, and let $A\in\La_{s}(X)$ and $B\in\La_{s}(Y)$ be such that ${\rm sp}(A)\cup{\rm sp}(B)\subseteq \R$. Then
\begin{align}\label{besov estimate}
\norm{f(B)S-Sf(A)}_{\Ic}\leq C\funA\funB\norm{f}_{\Br^{1}_{\infty,1}(\R)}\!\norm{BS-SA}_{\Ic}
\end{align}
for all $S\in \La(X,Y)$ such that $BS-SA\in \Ic$. In particular, if $X=Y$ and $B-A\in \Ic$,
\begin{align*}
\norm{f(B)-f(A)}_{\Ic}\leq C\funA\funB\norm{f}_{\Br^{1}_{\infty,1}(\R)}\!\norm{B-A}_{\Ic}.
\end{align*}
\end{corollary}

In the case where the Banach ideal $\Ic$ is the space $\La(X,Y)$ of all bounded operators from $X$ to $Y$, we obtain the following corollary.

\begin{corollary}\label{besov perturbation inequality1}
There exists a universal constant $C\geq 0$ such that the following holds. Let $X$ and $Y$ be separable Banach spaces such that either $X$ or $Y$ has the bounded approximation property. Let $f\in\Br^{1}_{\infty,1}(\R)$, and let $A,B\in\La_{s}(X)$ be such that ${\rm sp}(A)\cup{\rm sp}(B)\subseteq \R$. Then
\begin{align}\label{besov estimate1}
\norm{f(B)S-Sf(A)}_{\La(X,Y)}\leq C\funA\funB\norm{f}_{\Br^{1}_{\infty,1}(\R)}\!\norm{BS-SA}_{\La(X,Y)}
\end{align}
for all $S\in\La(X,Y)$. In particular, if $X=Y$ then
\begin{align*}
\norm{f(B)-f(A)}_{\La(X)}\leq C\funA\funB\norm{f}_{\Br^{1}_{\infty,1}(\R)}\!\norm{B-A}_{\La(X)}.
\end{align*}
\end{corollary}

\begin{remark}\label{local besov}
Corollaries \ref{besov perturbation inequality} and \ref{besov perturbation inequality1} yield global estimates, in the sense that \eqref{besov estimate} and \eqref{besov estimate1} hold for all scalar type operators $A$ and $B$ with
real spectrum, and the constant in the estimate depends on $A$ and $B$ only through their spectral constants $\funA$ and $\funB$. Local estimates follow if $f\in\B(\R)$ is contained in the Besov class locally. More precisely, given scalar type operators $A\in\La_{s}(X)$ and $B\in\La_{s}(Y)$ with real spectrum, suppose there exists $g\in\Br^{1}_{\infty,1}(\R)$ with $g(s)=f(s)$ for all $s\in{\rm sp}(A)\cup{\rm sp}(B)$. Then (with notation as in Corollary \ref{besov perturbation inequality})
\begin{align}\label{local besov inequality}
\norm{f(B)S-Sf(A)}_{\Ic}\leq C\funA\funB\norm{g}_{\Br^{1}_{\infty,1}(\R)}\!\norm{BS-SA}_{\Ic}
\end{align}
for all $S\in \La(X,Y)$ such that $BS-SA\in \Ic$. This follows directly from Theorem \ref{general perturbation inequality}.
\end{remark}

\section{Spaces with an unconditional basis}\label{unconditional basis spaces}

In this section we prove some results for specific scalar type operators, namely operators which are diagonalizable with respect to an unconditional Schauder basis. These results will be used in later sections. In this section we assume all spaces to be infinite-dimensional, but the results and proofs carry over directly to finite-dimensional spaces. This will be used in Section \ref{finite-dimensional spaces}.

\subsection{Diagonalizable operators}\label{diagonalizable operators}

Let $X$ be a Banach space with an unconditional Schauder basis $\{e_j\}_{j=1}^\infty\subseteq X$. For $j\in \mathbb N$, let $\Pp_{j}\in\La(X)$ be the projection given by $\Pp_{j}(x):=x_{j}e_{j}$ for all $x=\sum_{k=1}^\infty x_k e_k\in X$. 

\begin{assumption}\label{basis assumption}
For simplicity, we assume in this section that $\norm{\sum_{j\in N}\mathcal{P}_{j}}_{\La(X)}=1$ for all non-empty $N\subseteq \N$. This condition is satisfied in the examples we consider in later sections, and simplifies the presentation. For general bases one merely gets additional constants in the results.
\end{assumption}

An operator $A\in\La(X)$ is \emph{diagonalizable} (with respect to $\{e_j\}_{j=1}^\infty$) if there exists $U\in\La(X)$ invertible and a sequence $\{\lambda_j\}_{j=1}^\infty\in \ell_{\infty}$ of complex numbers such that
\begin{align}\label{diagonal_infinite}
UAU^{-1}x=\sum_{j=1}^{\infty}\lambda_{j}\Pp_{j}x\qquad(x\in X),
\end{align}
where the series converges since $\{e_k\}_{k=1}^\infty$ is unconditional, cf.~\cite[Lemma 16.1]{Singer1970}. In this case $A$ is a scalar type operator, with point spectrum equal to $\{\lambda_{j}\}_{j=1}^{\infty}$, $\spA=\overline{\left\{\lambda_{j}\right\}_{j=1}^{\infty}}$ and spectral measure $E$ given by
\begin{align}\label{spectral measure diagonalizable}
E(\sigma)=\sum_{\lambda_{j}\in\sigma}U^{-1}\Pp_{j}U
\end{align}
for $\sigma\subseteq \C$ Borel. The set of all diagonalizable operators on $X$ is denoted by
$\La_{\ud}(X)$. We do not explicitly mention the basis $\{e_j\}_{j=1}^\infty$ with respect to which an operator is diagonalizable, since this basis will be fixed throughout. Often we write $A\in \La_{\ud}(X, \{\lambda_j\}_{j=1}^\infty, U)$ in order to identify the operator $U$ and the sequence $\{\lambda_j\}_{j=1}^\infty$ from above. For $A\in\La_{\ud}(X, \{\lambda_j\}_{j=1}^\infty, U)$ and $f\in \B(\C)$, it follows from \eqref{functional calculus} that
\begin{align}\label{functional calculus matrices_infinite}
f(A)=U^{-1}\Big(\sum_{j=1}^{\infty}f(\lambda_{j})\Pp_{j}\Big)U.
\end{align}
Since any Banach space with a Schauder basis is separable and has the bounded approximation property, we can apply the results from the previous section to diagonalizable operators, and we obtain for instance the following.

\begin{corollary}\label{besov perturbation inequality diagonalizable}
There exists a universal constant $C\geq 0$ such that the following holds. Let $X$ and $Y$ be Banach spaces with unconditional Schauder bases, and let $\Ic$ be a Banach ideal in $\La(X,Y)$ with the strong convex compactness property. Let $f\in\Br^{1}_{\infty,1}(\R)$, and let $A\in\La_{\ud}(X)$ and $B\in\La_{\ud}(Y)$ be such that ${\rm sp}(A)\cup{\rm sp}(B)\subseteq \R$. Then
\begin{align*}
\norm{f(B)S-Sf(A)}_{\Ic}\leq C\funA\funB\norm{f}_{\Br^{1}_{\infty,1}(\R)}\!\norm{BS-SA}_{\Ic}
\end{align*}
for all $S\in \La(X,Y)$ such that $BS-SA\in \Ic$. In particular, if $X=Y$ and $B-A\in \Ic$,
\begin{align*}
\norm{f(B)-f(A)}_{\Ic}\leq C\funA\funB\norm{f}_{\Br^{1}_{\infty,1}(\R)}\!\norm{B-A}_{\Ic}.
\end{align*}
\end{corollary}

Since this result does not apply to the absolute value function (and neither does the more general Theorem \ref{general perturbation inequality}), and because of the importance of the absolute value function, we now study Lipschitz estimates for more general functions.

Let $Y$ be a Banach space with an unconditional Schauder basis $\{f_{k}\}_{k=1}^\infty\subseteq Y$, and let the projections $\mathcal{Q}_{k}\in\La(Y)$ be given by $\mathcal{Q}_{k}(y):=y_{k}f_{k}$ for all $y=\sum_{l=1}^\infty y_{l} f_{l}\in Y$ and $k\in\N$. Let $\lambda=\{\lambda_{j}\}_{j=1}^{\infty}$ and $\mu=\{\mu_{k}\}_{k=1}^{\infty}$ be sequences of complex numbers, and let $\ph:\C^{2}\to \C$. For $n\in\N$, define $T_{\ph,n}^{\lambda,\mu}\in\La(\La(X,Y))$ by
\begin{align}\label{double operator integral discrete1}
T^{\lambda,\mu}_{\ph,n}(S):=\sum_{j,k=1}^{n}\ph_{f}(\lambda_{j},\mu_{k})\Qq_{k}S\Pp_{j}\qquad(S\in\La(X,Y)).
\end{align}
Note that $T^{\lambda,\mu}_{\ph,n}\in\La(\Ic)$ for each Banach ideal $\Ic$ in $\La(X,Y)$. 

Let $f\in \B(\C)$ and extend the divided difference from \eqref{divided difference}, given by
\begin{align*}
\ph_{f}(\lambda_{1},\lambda_{2}):=\frac{f(\lambda_{2})-f(\lambda_{1})}{\lambda_{2}-\lambda_{1}}
\end{align*}
for $(\lambda_{1},\lambda_{2})\in\C^{2}$ with $\lambda_{1}\neq\lambda_{2}$, to a function $\ph_{f}:\C^{2}\rightarrow \C$. 

\begin{lemma}\label{write as schur multiplier}
Let $X$ and $Y$ be Banach spaces with unconditional Schauder bases, and let $\Ic$ be a Banach ideal in $\La(X,Y)$. Let $\lambda=\{\lambda_{j}\}_{j=1}^{\infty}$ and $\mu=\{\mu_{k}\}_{k=1}^{\infty}$ be sequences of complex numbers, and let $A\in\La_{\ud}(X,\lambda, U)$, $B\in\La_{\ud}(Y,\mu,V)$, $f\in\B(\C)$ and $n\in\N$. Then
\begin{align*}
\norm{f(B)S_{n}-S_{n}f(A)}_{\Ic}\leq \|U\|_{\La(X)}\|V^{-1}\|_{\La(Y)}\norm{T^{\lambda,\mu}_{\ph_{f},n}(V(BS-SA)U^{-1})}_{\Ic}
\end{align*}
for all $S\in \Ic$ such that $BS-SA\in\Ic$, where $S_{n}:=\sum_{j,k=1}^{n}V^{-1}\Qq_{k}VSU^{-1}\Pp_{j}U$.
\end{lemma}
\begin{proof}
Let $S\in\Ic$ be such that $BS-SA\in\Ic$. For the duration of the proof write $P_{j}:=U^{-1}\Pp_{j}U\in\La(X)$ and $Q_{k}:=V^{-1}\Qq_{k}V\in\La(Y)$ for $j,k\in\N$. By \eqref{functional calculus matrices_infinite}, and using that $P_{j}P_{k}=0$ and $Q_{j}Q_{k}=0$ for $j\neq k$,
\begin{align*}
f(B)S_{n}-S_{n}f(A)&= \sum_{k=1}^{\infty}f(\mu_{k})Q_{k}\Big(\sum_{i,l=1}^{n}Q_{l}SP_{i}\Big)-\sum_{j=1}^{\infty}f(\lambda_{j})\Big(\sum_{i,l=1}^{n}Q_{l}SP_{i}\Big)P_{j}\\
&=\sum_{j,k=1}^{n}(f(\mu_{k})-f(\lambda_{j}))Q_{k}SP_{j}\\
&=\sum_{j,k=1}^{n}\sum_{\mu_{k}\neq \lambda_{j}}\frac{f(\mu_{k})-f(\lambda_{j})}{\mu_{k}-\lambda_{j}}(\mu_{k}Q_{k}SP_{j}-\lambda_{j}Q_{k}SP_{j})\\
&=\sum_{j,k=1}^{n}\ph_{f}(\lambda_{j},\mu_{k})Q_{k}\Big(\Big(\sum_{l=1}^{\infty}\mu_{l}Q_{l}\Big)S-S\Big(\sum_{i=1}^{\infty}\lambda_{i}P_{i}\Big)\Big)P_{j}\\
&=\sum_{j,k=1}^{n}\ph_{f}(\lambda_{j},\mu_{k})Q_{k}(BS-SA)P_{j}\\
&=V^{-1}T^{A,B}_{\ph_{f}}(V(BS-SA)U^{-1})U.
\end{align*}
where we have used that $BS-SA\in M_{n}$. Now use the ideal property of $\Ic$ to conclude the proof.
\end{proof}

For a sequence $\lambda$ of complex numbers and $A\in\La_{\ud}(X,\lambda,\mathcal{U})$, define
\begin{align}\label{definition diagonalizable constant}
\diagA:=\inf\left\{\left.\|U\|_{\La(X)}\|U^{-1}\|_{\La(X)}\right| A\in\La_{\ud}(X,\lambda,U)\right\}.
\end{align}
We will call $\diagA$ the \emph{diagonalizability constant of $A$}. Using the unconditionality of the Schauder basis of $X$ and Assumption \ref{basis assumption}, one can show that $\diagA$ does not depend on the specific ordering of the sequence $\lambda$. Since the sequence $\lambda$ is, up to ordering, uniquely determined by $A$ (it is the point spectrum of $A$), $\diagA$ only depends on $A$. Moreover, by Assumption \ref{basis assumption} and \eqref{spectral measure diagonalizable}, $\norm{E(\sigma)}_{\La(X)}\leq \norm{U^{-1}}_{\La(X)}\norm{U}_{\La(X)}$ for all $\sigma\subseteq \C$ Borel and $U\in\La(X)$ such that $A\in\La_{\ud}(X,\lambda,U)$, where $E$ is the spectral measure of $A$. Hence 
\begin{align}\label{spectral estimate}
\specA\leq \diagA,
\end{align}
where $\specA$ is the spectral constant of $A$ from Section \ref{scalar type operators}. 

We now derive commutator estimates for $A$ and $B$ in the operator norm, under a boundedness assumption which will be verified for specific $X$ and $Y$ in later sections. 

\begin{proposition}\label{commutator estimates diagonalizable}
Let $X$ and $Y$ be Banach spaces with unconditional Schauder bases, $A\in\La_{\ud}(X,\lambda, U)$, $B\in\La_{\ud}(Y,\mu,V)$ and $f\in\B(\C)$. Suppose that 
\begin{align}\label{boundedness condition}
C:=\sup_{n\in\N}\norm{T^{\lambda,\mu}_{\ph_{f},n}}_{\La(\La(X,Y))}<\infty.
\end{align}
Then
\begin{align*}
\norm{f(B)S-Sf(A)}_{\La(X,Y)}\leq C\diagA\diagB\norm{BS-SA}_{\La(X,Y)}
\end{align*}
for all $S\in\La(X,Y)$.
\end{proposition}
\begin{proof}
Let $S\in\La(X,Y)$ and let $S_{n}\in\La(X,Y)$ be as in Lemma \ref{write as schur multiplier} for $n\in\N$. It is straightforward to show that, for each $x\in X$, $S_{n}x\to Sx$ as $n\to \infty$. Hence $f(B)S_{n}x-S_{n}f(A)x\to f(B)Sx-Sf(A)x$ as $n\to \infty$, for each $x\in X$. Lemma \ref{write as schur multiplier} and \eqref{boundedness condition} now yield
\begin{align*}
\norm{f(B)S-Sf(A)}_{\La(X,Y)}&\leq \limsup_{n\to \infty}\norm{f(B)S_{n}-S_{n}f(A)}_{\La(X,Y)}\\
&\leq C\|U\|\|V^{-1}\|\|V(BS-SA)U^{-1}\|_{\La(X,Y)}\\
&\leq C\|U\|\|U^{-1}\|\|V\|\|V^{-1}\|\,\|BS-SA\|_{\La(X,Y)}.
\end{align*}
Taking the infimum over $U$ and $V$ concludes the proof.
\end{proof}

\begin{remark}\label{other ideals}
Proposition \ref{commutator estimates diagonalizable} also holds for more general Banach ideals in $\La(X,Y)$. Indeed, let $\Ic$ be a Banach ideal in $\La(X,Y)$ with the property that, if $\{S_{m}\}_{m=1}^{\infty}\subseteq\Ic$ is an $\Ic$-bounded sequence which SOT-converges to some $S\in\La(X,Y)$ as $m\to \infty$, then $S\in\Ic$ with $\norm{S}_{\Ic}\leq \limsup_{m\to\infty}\norm{S_{m}}_{\Ic}$. With notation as in Proposition \ref{commutator estimates diagonalizable}, if
\begin{align*}
C:=\sup_{n\in\N}\norm{T^{\lambda,\mu}_{\ph_{f},n}}_{\La(\Ic)}<\infty
\end{align*}
then 
\begin{align*}
\norm{f(B)S-Sf(A)}_{\Ic}\leq C\diagA\diagB\norm{BS-SA}_{\Ic}
\end{align*}
for all $S\in\La(X,Y)$ such that $BS-SA\in\Ic$. This follows directly from the proof of Proposition \ref{commutator estimates diagonalizable}.
\end{remark}

\subsection{Estimates for the absolute value function}\label{absolute value function}

It is known that Lipschitz estimates for the absolute value function are related to estimates for so-called triangular truncation operators. For example, in \cite{Kosaki1992} and \cite{DDPS} it was shown that the boundedness of the standard triangular truncation on many operator spaces is equivalent to Lipschitz estimates for the absolute value function. We prove that, in our setting, triangular truncation operators are also related to Lipschitz estimates for $f$ the absolute value function. We do so by relating the assumption in \eqref{boundedness condition} to so-called triangular truncation operators associated to sequences. We will then bound the norms of these operators in later sections for specific $X$ and $Y$.

Let $\lambda=\{\lambda_{j}\}_{j=1}^{\infty}$ and $\mu=\{\mu_{k}\}_{k=1}^{\infty}$ be sequences of real numbers, and let $X$ and $Y$ be as before. For $n\in\N$ define $T_{\Delta,n}^{\lambda,\mu}\in\La(\La(X,Y))$ by
\begin{align}\label{sequence truncation operator}
T_{\Delta,n}^{\lambda,\mu}(S):=\sum_{j,k=1}^{n}\sum_{\mu_{k}\leq \lambda_{j}}\Qq_{k}S\Pp_{j}\qquad(S\in\La(X,Y)).
\end{align}
We call $T_{\Delta}^{A,B}$ the \emph{triangular truncation associated to $\lambda$ and $\mu$}.

For $f(t):=\abs{t}$ for $t\in\R$, define $\ph_{f}:\C^{2}\to \C$ by
\begin{align}\label{psi_f_2}
\ph_{f}(\lambda_{1},\lambda_{2}):= \left\{\begin{array}{ll}
\frac{\abs{\lambda_1}-\abs{\lambda_2}}{\lambda_1-\lambda_2}&\textrm{if }\lambda_{1}\neq\lambda_{2}\\
1&\textrm{otherwise}
\end{array}.\right.
\end{align}  
The following result relates the norm of $T_{\ph_{f},n}^{\lambda,\mu}$ to that of $T_{\Delta,n}^{\lambda,\mu}$.

\begin{proposition}\label{truncation estimate_inf}
There exists a universal constant $C\geq 0$ such that the following holds. Let $X$ and $Y$ be Banach spaces with unconditional Schauder bases and let $\Ic$ be a Banach ideal in $\La(X,Y)$ with the strong convex compactness property. Let $\lambda$ and $\mu$ be bounded sequences of real numbers. Let $f(t):=\abs{t}$ for $t\in\R$. Then
\begin{align*}
\norm{T^{\lambda,\mu}_{\ph_{f},n}(S)}_{\Ic}\leq C\Big(\|S\|_{\Ic}+\|T_{\Delta,n}^{\lambda,\mu}(S)\|_{\Ic}\Big)
\end{align*}
for all $n\in\N$ and $S\in\Ic$. In particular, if $\sup_{n\in\N}\|T_{\Delta,n}^{\lambda,\mu}(S)\|_{\La(\La(X,Y))}<\infty$ then \eqref{boundedness condition} holds.
\end{proposition}
\begin{proof}
Let $n\in\N$ and $S\in \Ic$, and write $\lambda=\{\lambda_{j}\}_{j=1}^{\infty}$ and $\mu=\{\mu_{k}\}_{k=1}^{\infty}$. Throughout the proof we will only consider $\lambda_{j}$ and $\mu_{k}$ for $1\leq j,k\leq n$, but to simplify the presentation we will not mention this explicitly from here on. We can decompose $T_{\ph_{f},n}^{\lambda,\mu}(S)$ as
\begin{align*}
T_{\ph_{f},n}^{\lambda,\mu}(S)=&\sum_{\lambda_{k},\mu_{k}\geq 0}\Qq_{k}S\Pp_{j}-\sum_{\mu_{k}<0<\lambda_{j}}\frac{\mu_{k}+\lambda_{j}}{\mu_{k}-\lambda_{j}}\Qq_{k}S\Pp_{j}+\\ 
&\sum_{\lambda_{j}<0<\mu_{k}}
\frac{\mu_{k}+\lambda_{j}}{\mu_{k}-\lambda_{j}}\Qq_{k}S\Pp_{j}
-\sum_{\lambda_{k},\mu_{k}\leq 0}\Qq_{k}S\Pp_{j}+\sum_{\lambda_{k},\mu_{k}=0}\Qq_{k}S\Pp_{j}.
\end{align*}
Note that some of these terms may be zero. By the ideal property of $\Ic$ and Assumption \ref{basis assumption}, 
\begin{align}\label{subset estimate}
\Big\|\sum_{\lambda_{j},\mu_{k}\geq0}\Qq_{k}S\Pp_{j}\Big\|_{\Ic}\leq \Big\|\sum_{\mu_{k}\geq0}\Qq_{k}\Big\|_{\La(Y)}\norm{S}_{\Ic}\Big\|\sum_{\lambda_{j}\geq0}\Pp_{j}\Big\|_{\La(X)}\leq \norm{S}_{\Ic}.
\end{align}
Similarly, $\norm{\sum_{\lambda_{k},\mu_{k}\leq 0}\Qq_{k}S\Pp_{j}}_{\Ic}$ and $\norm{\sum_{\lambda_{k},\mu_{k}=0}\Qq_{k}S\Pp_{j}}_{\Ic}$ are each bounded by $\norm{S}_\Ic$. To bound the other terms it is sufficient to show that
\begin{align*}
\norm{\sum_{\lambda_{j},\mu_{k}>0}\frac{\mu_{k}-\lambda_{j}}{\mu_{k}+\lambda_{j}}\Qq_{k}S\Pp_{j}}_\Ic\leq C'\Big(\norm{S}_{\Ic}+\norm{{T}_{\Delta,n}^{\lambda,\mu}(S)}_{\Ic}\Big)
\end{align*}
for some universal constant $C'\geq 0$. Indeed, replacing $\lambda$ by $-\lambda$ and $\mu$ by $-\mu$ then yields the desired conclusion. Let
\begin{align*}
\Phi(S):=\sum_{\lambda_{j},\mu_{k}>0}\frac{\mu_{k}-\lambda_{j}}{\mu_{k}+\lambda_{j}}\Qq_{k}S\Pp_{j},
\end{align*}
and define $g\in\W^{1,2}(\R)$ by $g(t):=\frac{2}{e^{\abs{t}}+1}$
for $t\in\R$. Note that $\Phi(S)$ is equal to
\begin{align*}
\sum_{0<\mu_{k}\leq\lambda_{j}}\left(g\left(\log\tfrac{\lambda_{j}}{\mu_{k}}\right)-1\right)\Qq_{k}S\Pp_{j}
+\sum_{0<\lambda_{j}<\mu_{k}}\left(1-g\left(\log\tfrac{\lambda_{j}}{\mu_{k}}\right)\right)\Qq_{k}S\Pp_{j}.
\end{align*}
Now let $\psi_{g}:\C^{2}\rightarrow \C$ be as in \eqref{psi_f}, and let $A:=\sum_{j=1}^{\infty}\lambda_{j}\Pp_{j}\in\La(X)$ and $B:=\sum_{k=1}^{\infty}\mu_{k}\Qq_{k}\in\La(Y)$. Let $T_{\psi_{g}}^{A,B}$ be as in \eqref{double operator integral}. One can check that
\begin{align*}
\Phi(S)=&T_{\psi_{g}}^{A,B}(T_{\Delta,n}^{\lambda,\mu}(S))-\sum_{\lambda_{j},\mu_{k}>0}\Qq_{k}T_{\Delta,n}^{\lambda,\mu}(S)\Pp_{j}+\\
&\sum_{\lambda_{j},\mu_{k}>0}\Qq_{k}(S-T_{\Delta,n}^{\lambda,\mu}(S))\Pp_{j}
-T_{\psi_{g}}^{A,B}(S-{T}_{\Delta,n}^{\lambda,\mu}(S)).
\end{align*}
Any Banach space with a Schauder basis is separable and has the bounded approximation property, hence Lemma \ref{sobolev condition} and Proposition \ref{operator integral well defined} yield 
\begin{align*}
\norm{T_{\psi_{g}}^{A,B}(T_{\Delta,n}^{\lambda,\mu}(S))}_{\Ic}\leq 16\sqrt{2}\,\specA\specB\norm{g}_{\W^{1,2}(\R)}\norm{T_{\Delta,n}^{\lambda,\mu}(S))}_{\Ic}.
\end{align*}
By \eqref{spectral estimate}, $\specA=\specB=1$. Similarly,
\begin{align*}
\norm{T_{\psi_{g}}^{A,B}(S-T_{\Delta,n}^{\lambda,\mu}(S))}_{\Ic}\leq 16\sqrt{2}\norm{g}_{\W^{1,2}(\R)}\Big(\Big\|S\Big\|_{\Ic}+\Big\|T_{\Delta,n}^{\lambda,\mu}(S))\Big\|_{\Ic}\Big).
\end{align*}
By the same arguments as in \eqref{subset estimate},
\begin{align*}
\Big\|\sum_{\lambda_{j},\mu_{k}>0}\Qq_{k}T_{\Delta,n}^{\lambda,\mu}(S)\Pp_{j}\Big\|_{\Ic}+\Big\|\sum_{\lambda_{j},\mu_{k}>0}\Qq_{k}(S-T_{\Delta,n}^{\lambda,\mu}(S))\Pp_{j}\Big\|\leq 2\Big\|S\Big\|_{\Ic}+\Big\|T_{\Delta,n}^{\lambda,\mu}(S)\Big\|_{\Ic}.
\end{align*}
Combining all these estimates yields
\begin{align*}
\norm{\Phi(S)}_\Ic\leq
\Big(2+32\sqrt{2}\norm{g}_{\W^{1,2}(\R)}\Big)
\Big(\Big\|S\Big\|_{\Ic}+\Big\|{T}_{\Delta,n}^{\lambda,\mu}(S)\Big\|_{\Ic}\Big),
\end{align*}
as desired.
\end{proof}

\section{The absolute value function on $\La(\ell_{p},\ell_{q})$}\label{lp-spaces}

In this section we study the absolute value function on the space $\La(\ell_{p},\ell_{q})$. We show that the absolute value function is operator Lipschitz on $\La(\ell_{p},\ell_{q})$ for $p<q<\infty$, on $\La(\ell_{1})$ and on $\La(\ce_{0})$. We also obtain results for $p\geq q$. 

The key idea of the proof is entirely different from the techniques used in \cite{Davies},\cite{PWS02}, \cite{DDPS} and \cite{Kosaki1992}, which are based on a special geometric property of the reflexive Schatten von Neumann ideals (the UMD-property), a property which $\La(\ell_{p},\ell_{q})$ does not have. Instead, we prove our results by relating estimates for the operators from \eqref{sequence truncation operator} to the standard triangular truncation operator, defined in \eqref{classical_tr_tr} below. For this we use the theory of Schur multipliers on $\La(\ell_{p},\ell_{q})$ developed in \cite{Ben}. We then appeal to results from \cite{Ben1} about the boundedness of the standard triangular truncation on $\La(\ell_{p},\ell_{q})$.

\subsection{Schur multipliers}

For $p\in[1,\infty)$, let $\{e_{j}\}_{j=1}^{\infty}$ be the standard Schauder basis of $\ell_{p}$, with the corresponding projections $\Pp_{j}(x):=x_{j}e_{j}$ for $x=\sum_{k=1}^{\infty}x_{k}e_{k}$ and $j\in\N$. We consider this basis and the corresponding projections on all $\ell_{p}$-spaces simultaneously, for simplicity of notation. Note that Assumption \ref{basis assumption} is satisfied for this basis. For $q\in[1,\infty]$, any operator $S\in \La(\ell_{p},\ell_{q})$ can be represented by an infinite matrix $\tilde{S}=\{s_{jk}\}_{j,k=1}^\infty$, where $s_{jk}:=(S(e_k),e_j)$ for $j,k\in\N$. For an infinite matrix $M=\{m_{jk}\}_{j,k=1}^\infty$ the
product $M\ast \tilde{S}:=\{m_{jk}s_{jk}\}$ is the \emph{Schur
product} of the matrices $M$ and $\tilde{S}$. The matrix $M$ is a
\emph{Schur multiplier} if the mapping $\tilde{S}\mapsto M\ast \tilde{S}$ is a bounded operator on $\La(\ell_{p},\ell_{q})$. Throughout, we identify Schur multipliers with their corresponding operators.

The notion of a Schur multiplier is a discrete version of a double operator integral (for details see e.g.~\cite{PST1, SST-survey}). Schur multipliers on the space $\La(\ell_{p},\ell_{q})$ are also called $(p,q)$-multipliers.
We denote by $\mathcal M(p,q)$ the Banach space of $(p,q)$-multipliers with the norm
\begin{align*}
\norm{M}_{(p,q)}:=\sup\left\{\norm{M\ast \tilde{S}}_{\La(\ell_{p},\ell_{q})}\left| \norm{S}_{\La(\ell_{p},\ell_{q})}\leq 1\right.\right\}.
\end{align*}

\begin{remark}\label{c0}
We will also consider $(p,q)$-multipliers $M$ for $p=\infty$ and $q\in[1,\infty]$. Any operator $S\in\La(\ce_{0},\ell_{q})$ corresponds to an infinite matrix $\tilde{S}=\{s_{jk}\}_{j,k=1}^\infty$, and $M$ is said to be a $(\infty,q)$-multiplier if the mapping $S\mapsto M\ast \tilde{S}$ is a bounded operator on $\La(\ce_{0},\ell_{q})$, and we define the Banach space $\mathcal{M}(\infty,q)$ in the obvious way. Below we will often not explicitly distinguish the case $p=\infty$ from $1\leq p<\infty$, in order to keep the presentation simple. The reader should keep in mind that for $p=\infty$ the space $\ell_{p}$ should be replaced by $\ce_{0}$.
\end{remark} 

\begin{remark}\label{rem_max_entry}
It is straightforward to see that $\norm{M}_{(p,q)}\geq \sup_{j,k\in\N}\abs{m_{j,k}}$ for all $p,q\in[1,\infty]$ and $M\in\mathcal{M}(p,q)$.
\end{remark}

For $p,q\in[1,\infty]$ and $S\in\La(\ell_{p},\ell_{q})$, define
\begin{align}\label{classical_tr_tr}
\Tr_{\Delta}(S):=\sum_{k\leq j}\Pp_{k}S\Pp_{j},
\end{align}
which is a well-defined element of $\La(\ell_{r},\ell_{s})$ for suitable $r,s\in[1,\infty]$, by Proposition \ref{Pil-Ben} below. The operator $\Tr_{\Delta}$ is the \emph{(standard) triangular truncation}. This operator can be identified with the following Schur multiplier. Let $T_\Delta'=\{t'_{jk}\}_{j,k=1}^\infty$ be a matrix given by
$t'_{jk}=1$ for $k\le j$ and $t'_{jk}=0$ otherwise.
It is clear that $\Tr_{\Delta}$ extends to a bounded linear operator on $\La(\ell_{p},\ell_{q})$ if and only if $T'_\Delta$ is a $(p,q)$-multiplier. 
For $n\in\N$ and $r,s\in[1,\infty]$ we will consider the operators $\Tr_{\Delta,n}\in\La(\La(\ell_{p},\ell_{q}),\La(\ell_{r},\ell_{s}))$, given by 
\begin{align*}
\Tr_{\Delta,n}(S):=\sum_{1\leq k\leq j\leq n}\Pp_{k}S\Pp_{j}\qquad(S\in \La(\ell_{p},\ell_{q}).
\end{align*}

The dependence of the $(p,q)$-norm of $\Tr_{\Delta}$ on the indices $p$ and $q$ was determined in \cite{Ben1} and \cite{KP} (see also \cite{ST}), and is as follows. 

\begin{proposition}\label{Pil-Ben} 
Let $p,q\in[1,\infty]$. Then the following statements hold.
\begin{itemize}
\item[$(i)$]\cite[Theorem 5.1]{Ben1} If $p<q$, $1=p=q$ or $p=q=\infty$, then $\Tr_{\Delta}\in \mathcal{M}(p,q)$.
\item[$(ii)$]\cite[Proposition 1.2]{KP} If $1\neq p\geq q\neq \infty$, then there is a constant $C>0$ such that 
\begin{align*}
\norm{\Tr_{\Delta,n}}_{\La(\La(\ell_{p},\ell_{q}))}\geq C \ln n
\end{align*}
for all $n\in\N$. 
\item[$(iii)$]\cite[Theorem 5.2]{Ben1} If $1\neq p\geq q\neq\infty$, then for each $s>q$ and $r<p$, 
\begin{align*}
\Tr_{\Delta}:\La(\ell_{p},\ell_{q})\to\La(\ell_{p},\ell_{s})\quad\textrm{and}\quad \Tr_{\Delta}:\La(\ell_{p},\ell_{q})\to\La(\ell_{r},\ell_{q})
\end{align*}
are bounded.
\end{itemize}
\end{proposition}

We will also need the following result, a generalization of \cite[Theorem 4.1]{Ben}. For a matrix $M=\{m_{jk}\}_{j,k=1}^{\infty}$, let $\widetilde{M}=\{\widetilde{m}_{jk}\}_{j,k=1}^{\infty}$ be obtained from $M$ by repeating the first column, i.e., $\widetilde{m}_{j1}=m_{j1}$ and $\widetilde{m}_{jk}=m_{j(k-1)}$ for $j\in\N$ and $k\geq 2$. 

\begin{proposition}\label{Th_ColRepeating} 
Let $p,q,r,s\in[1,\infty]$ with $r\leq p$. Let $M=\{m_{jk}\}_{j,k=1}^{\infty}$ be such that $S\mapsto M\ast S$ is a bounded mapping $\La(\ell_{p},\ell_{q})\to \La(\ell_{r},\ell_{s})$. Then $S\mapsto \widetilde{M}\ast S$ is also a bounded mapping $\La(\ell_{p},\ell_{q})\to \La(\ell_{r},\ell_{s})$, with 
\begin{align*}
\norm{M}_{\La(\La(\ell_{p},\ell_{q}),\La(\ell_{r},\ell_{s}))}=\norm{\widetilde{M}}_{\La(\La(\ell_{p},\ell_{q}),\La(\ell_{r},\ell_{s}))}.
\end{align*}
In particular, if $M\in\mathcal{M}(p,q)$ then $\widetilde{M}\in\mathcal{M}(p,q)$ with $\norm{M}_{(p,q)}=\norm{\widetilde{M}}_{(p,q)}$.
\end{proposition}
\begin{proof}
The proof is almost identical to that of \cite[Theorem 4.1]{Ben}, and the condition $r\leq p$ is used to ensure that $\abs{x_{1}}^{p}+\abs{x_{2}}^{p}\leq (\abs{x_{1}}^{r}+\abs{x_{2}}^{r})^{p/r}$ for all $x_{1},x_{2}\in\C$ (with the obvious modification for $p=\infty$ or $r=\infty$).
\end{proof}

\begin{remark}\label{rem_on_Ben}
By considering the transpose $M'$ of a matrix $M$, and using that $M':\La(\ell_{q^{*}},\ell_{p^{*}})\to \La(\ell_{s^{*}},\ell_{r^{*}})$ with 
\begin{align*}
\norm{M'}_{\La(\La(\ell_{q^{*}},\ell_{p^{*}}),\,\La(\ell_{s^{*}},\ell_{r^{*}}))}=\norm{M}_{\La(\La(\ell_{p},\ell_{q}),\,\La(\ell_{r},\ell_{s}))},
\end{align*}
as is straightforward to check, one obtains from Proposition \ref{Th_ColRepeating} that for $q^{*}\geq s^{*}$, that is, for $s\leq q$, row repetitions do not alter the $\norm{\cdot}_{\La(\La(\ell_{p},\ell_{q}),\,\La(\ell_{r},\ell_{s}))}$-norm of a matrix. Moreover, since $\norm{S}_{\La(\ell_{p},\ell_{q})}$ is invariant under permutations of the columns and rows of $S\in\La(\ell_{p},\ell_{q})$, rearrangement of the rows and columns of $M\in\La(\La(\ell_{p},\ell_{q}),\,\La(\ell_{r},\ell_{s}))$ does not change the norm $\norm{M}_{\La(\La(\ell_{p},\ell_{q}),\,\La(\ell_{r},\ell_{s}))}$.
\end{remark}

The following lemma is crucial to our main results. 

\begin{lemma}\label{crucial lemma}
Let $p,q,r,s\in[1,\infty]$ with $r\leq p$ and $s\leq q$. Let $\lambda=\{\lambda_{j}\}_{j=1}^{\infty}$ and $\mu=\{\mu_{k}\}_{k=1}^{\infty}$ be sequences of real numbers. Then
\begin{align*}
\norm{T^{\lambda,\mu}_{\Delta, n}}_{\La(\La(\ell_{p},\ell_{q}),\,\La(\ell_{r},\ell_{s}))}\leq \norm{\Tr_{\Delta,n}}_{\La(\La(\ell_{p},\ell_{q}),\,\La(\ell_{r},\ell_{s}))}
\end{align*}
for all $n\in\N$.
\end{lemma}
\begin{proof}
Note that $T_{\Delta,n}^{\lambda,\mu}(S)=M\ast S$ for all $S\in\La(\ell_{p},\ell_{q})$, where $M=\{m_{jk}\}_{j,k=1}^{\infty}$ is such that $m_{jk}=1$ if $1\leq j,k\leq n$ and $\mu_{k}\leq \lambda_{j}$, and $m_{jk}=0$ otherwise. It suffices to prove that $\norm{M}_{\La(\La(\ell_{p},\ell_{q}),\,\La(\ell_{r},\ell_{s}))}\leq \norm{\Tr}_{\La(\La(\ell_{p},\ell_{q}),\,\La(\ell_{r},\ell_{s}))}$. Assume that $M$ is non-zero, otherwise the statement is trivial. By Remark \ref{rem_on_Ben}, rearrangement of the rows and columns of $M$ does not change its norm. Hence, we may assume that $\{\lambda_{j}\}_{j=1}^{n}$ and $\{\mu_{k}\}_{k=1}^{n}$ are decreasing. Now $M$ has the property that, if $m_{jk}=1$, then $m_{il}=1$ for all $i\leq j$ and $k\leq l\leq m_{2}$. By Proposition \ref{Th_ColRepeating} and Remark \ref{rem_on_Ben}, we may omit repeated rows and columns of $M$, and doing this repeatedly reduces $M$ to $\Tr_{\Delta,N}$ for some $1\leq N\leq n$. Noting that $\norm{\Tr_{\Delta,N}}_{\La(\La(\ell_{p},\ell_{q}),\,\La(\ell_{r},\ell_{s}))}\leq \norm{\Tr_{\Delta,n}}_{\La(\La(\ell_{p},\ell_{q}),\,\La(\ell_{r},\ell_{s}))}$ concludes the proof.
\end{proof}

\subsection{The cases $p<q$, $p=q=1$, $p=q=\infty$}\label{p<q}

We now combine the theory from the previous sections to deduce our main results. Throughout this section let $f(t):=\abs{t}$ for $t\in\R$. 

\begin{theorem}\label{truncation estimate_inf_lplq} 
Let $p,q\in[1,\infty]$ with $p<q$, $p=q=1$ or $p=q=\infty$. Then there exists a constant $C\geq 0$ such that the following holds (where $\ell_{\infty}$ should be replaced by $\ce_{0}$). Let $A\in\La_{\ud}(\ell_{p})$ and $B\in\La_{\ud}(\ell_{q})$ have real spectrum. Then
\begin{align*}
\norm{f(B)S-Sf(A)}_{\La(\ell_{p},\ell_{q})}\leq C\diagA\diagB \norm{BS-SA}_{\La(\ell_{p},\ell_{q})}
\end{align*}
for all $S\in \La(\ell_{p},\ell_{q})$.
\end{theorem}
\begin{proof}
Simply combine Propositions \ref{commutator estimates diagonalizable} and \ref{truncation estimate_inf} and Lemma \ref{crucial lemma} with Proposition \ref{Pil-Ben} $(i)$.
\end{proof}

We single out the specific case in Theorem \ref{truncation estimate_inf_lplq} where $p=q=1$ or $p=q=\infty$ and $S$ is the identity operator.

\begin{corollary}\label{lipschitz l1}
There exists a universal constant $C\geq 0$ such that
\begin{align*}
\norm{f(B)-f(A)}_{\La(\ell_{1})}\leq C\diagA\diagB \norm{B-A}_{\La(\ell_{1})}
\end{align*}
for all $A,B\in\La_{\ud}(\ell_{1})$ with real spectrum.
\end{corollary}

\begin{corollary}\label{lipschitz c0}
There exists a universal constant $C\geq 0$ such that
\begin{align*}
\norm{f(B)-f(A)}_{\La(\ce_{0})}\leq C\diagA\diagB \norm{B-A}_{\La(\ce_{0})}
\end{align*}
for all $A,B\in\La_{\ud}(\ce_{0})$ with real spectrum.
\end{corollary}

In the case of Theorem \ref{truncation estimate_inf_lplq} where $p=2$ or $q=2$, we can apply our results to compact self-adjoint operators. By the spectral theorem, any compact self-adjoint operator $A\in\La(\ell_{2})$ has an orthonormal basis of eigenvectors, and therefore $A\in\La_{\ud}(\ell_{2},\lambda,U)$ for some sequence $\lambda$ of real numbers and an isometry $U\in\La(\ell_{2})$. Thus Theorem \ref{truncation estimate_inf_lplq} yields the following corollaries.

\begin{corollary}\label{to l2}
Let $p\in[1,2)$. Then there exists a constant $C\geq 0$ such that the following holds. Let $A\in\La_{\ud}(\ell_{p})$ and let $B\in\La(\ell_{2})$ be compact and self-adjoint. Then
\begin{align*}
\norm{f(B)S-Sf(A)}_{\La(\ell_{p},\ell_{2})}\leq C\diagA \norm{BS-SA}_{\La(\ell_{p},\ell_{2})}
\end{align*}
for all $S\in \La(\ell_{p},\ell_{2})$.
\end{corollary}

\begin{corollary}\label{from l2}
Let $q\in(2,\infty]$. Then there exists a constant $C\geq 0$ such that the following holds (where $\ell_{\infty}$ should be replaced by $\ce_{0}$). Let $A\in\La(\ell_{2})$ be compact and self-adjoint, and let $B\in\La_{\ud}(\ell_{q})$. Then
\begin{align*}
\norm{f(B)S-Sf(A)}_{\La(\ell_{2},\ell_{q})}\leq C\diagB \norm{BS-SA}_{\La(\ell_{2},\ell_{q})}
\end{align*}
for all $S\in \La(\ell_{2},\ell_{q})$.
\end{corollary}

\begin{remark}\label{inclusion mapping}
Corollaries \ref{lipschitz l1} and \ref{lipschitz c0} show that the absolute value function is operator Lipschitz on $\ell_{1}$ and $\ce_{0}$, in the following sense. For fixed $M\geq 1$, there exists a constant $C\geq 0$ such that
\begin{align*}
\norm{f(B)-f(A)}\leq C \norm{B-A}
\end{align*}
for all diagonalizable operators $A$ and $B$ such that $\diagA, \diagB\leq M$, and $C$ is independent of $A$ and $B$. 

For $p<q$ a similar statement holds. Restricting $B\in\La_{\ud}(\ell_{q})$ and $f(B)\in\La(\ell_{q})$ to operators from $\ell_{p}$ to $\ell_{q}$, and letting $S$ be the inclusion mapping $\ell_{p}\hookrightarrow \ell_{q}$ in Theorem \ref{truncation estimate_inf_lplq}, one can suggestively write
\begin{align*}
\norm{f(B)-f(A)}_{\La(\ell_{p},\ell_{q})}\leq C\norm{B-A}_{\La(\ell_{p},\ell_{q})},
\end{align*}
for all $A\in\La_{\ud}(\ell_{p})$ and $B\in\La_{\ud}(\ell_{q})$ with $\diagA,\diagB\leq M$. This also applies to Corollaries \ref{to l2} and \ref{from l2}. 
\end{remark}

\subsection{The case $p\geq q$}

We now examine the absolute value function $f$ on $\La(\ell_{p},\ell_{q})$ for $p\geq q$, and obtain the following result. 

\begin{proposition}\label{case p geq q1}
Let $p,q\in(1,\infty]$ with $p\geq q$. Then for each $s<q$ there exists a constant $C\geq 0$ such that the following holds (where $\ell_{\infty}$ should be replaced by $\ce_{0}$). Let $A\in\La_{\ud}(\ell_{p},\lambda,U)$ and $B\in\La_{\ud}(\ell_{q},\mu,V)$ have real spectrum, and let $S\in\La(\ell_{p},\ell_{q})$ be such that $V(BS-SA)U^{-1}\in\La(\ell_{p},\ell_{s})$. Then
\begin{align*}
\norm{f(B)S-Sf(A)}_{\La(\ell_{p},\ell_{q})}\leq C\|U\|_{\La(\ell_{p})}\|V^{-1}\|_{\La(\ell_{q})}\|V(BS-SA)U^{-1}\|_{\La(\ell_{p},\ell_{s})}.
\end{align*}
In particular, if $p=q$ and $V(B-A)U^{-1}\in\La(\ell_{p},\ell_{s})$, then
\begin{align*}
\norm{f(B)-f(A)}_{\La(\ell_{p})}\leq C\|U\|_{\La(\ell_{p})}\|V^{-1}\|_{\La(\ell_{p})}\|V(B-A)U^{-1}\|_{\La(\ell_{p},\ell_{s})}.
\end{align*}
\end{proposition}
\begin{proof}
Let $R:=V(BS-SA)U^{-1}$. With notation as in Lemma \ref{write as schur multiplier}, 
\begin{align*}
\norm{f(B)S_{n}-S_{n}f(A)}_{\La(\ell_{p},\ell_{q})}\leq \|U\|_{\La(\ell_{p})}\|V^{-1}\|_{\La(\ell_{q})}\norm{T^{\lambda,\mu}_{\ph_{f},n}(R)}_{\La(\ell_{p},\ell_{q})}
\end{align*}
for each $n\in\N$. Proposition \ref{truncation estimate_inf}, Lemma \ref{crucial lemma} (with $p=r$ and with $q$ and $s$ interchanged) and Proposition \ref{Pil-Ben} $(iii)$ (with $q$ and $s$ interchanged) yield a constant $C'\geq 0$ such that
\begin{align*}
\norm{T^{\lambda,\mu}_{\ph_{f},n}(R)}_{\La(\ell_{p},\ell_{q})}\leq C'\left(\norm{R}_{\La(\ell_{p},\ell_{q})}+\norm{R}_{\La(\ell_{p},\ell_{s})}\right).
\end{align*}
Since $\La(\ell_{p},\ell_{s})\hookrightarrow \La(\ell_{p},\ell_{q})$ contractively,
\begin{align*}
\norm{f(B)S_{n}-S_{n}f(A)}_{\La(\ell_{p},\ell_{q})}\leq C\|U\|_{\La(\ell_{p})}\|V^{-1}\|_{\La(\ell_{q})}\|V(BS-SA)U^{-1}\|_{\La(\ell_{p},\ell_{s})}
\end{align*} 
for all $n\in\N$, where $C=2C'$. Finally, as in the proof of Proposition \ref{commutator estimates diagonalizable}, one lets $n$ tend to infinity to conclude the proof.
\end{proof}

In the same way, appealing to the second part of Proposition \ref{Pil-Ben} (iii), one deduces the following result.

\begin{proposition}\label{case p geq q2}
Let $p,q\in[1,\infty)$ with $p\geq q$. Then for each $r>p$ there exists a constant $C\geq 0$ such that the following holds (where $\ell_{\infty}$ should be replaced by $\ce_{0}$). Let $A\in\La_{\ud}(\ell_{p},\lambda,U)$ and $B\in\La_{\ud}(\ell_{q},\mu,V)$ have real spectrum, and let $S\in\La(\ell_{p},\ell_{q})$ be such that $V(BS-SA)U^{-1}\in\La(\ell_{r},\ell_{q})$. Then
\begin{align*}
\norm{f(B)S-Sf(A)}_{\La(\ell_{p},\ell_{q})}\leq C\|U\|_{\La(\ell_{p})}\|V^{-1}\|_{\La(\ell_{q})}\|V(BS-SA)U^{-1}\|_{\La(\ell_{r},\ell_{q})}.
\end{align*}
In particular, if $p=q$ and $V(B-A)U^{-1}\in\La(\ell_{r},\ell_{q})$, then
\begin{align*}
\norm{f(B)-f(A)}_{\La(\ell_{p})}\leq C\|U\|_{\La(\ell_{p})}\|V^{-1}\|_{\La(\ell_{p})}\|V(B-A)U^{-1}\|_{\La(\ell_{r},\ell_{q})}.
\end{align*}
\end{proposition}

We single out the case where $p=q=2$. Here we write $f(A)=\abs{A}$ for a normal operator $A\in\La(\ell_{2})$, since then $f(A)$ is equal to $\abs{A}:=\sqrt{A^{*}A}$. Note also that the following result applies in particular to compact self-adjoint operators. For simplicity of the presentation we only consider $\epsilon\in(0,1]$, it should be clear how the result extends to other $\epsilon>0$.

\begin{corollary}\label{case p=q=2}
For each $\epsilon\in(0,1]$ there exists a constant $C\geq 0$ such that the following holds. Let $A\in\La_{\ud}(\ell_{2},\lambda,U)$ and $B\in\La_{\ud}(\ell_{2},\mu,V)$ be self-adjoint, with $U$ and $V$ unitaries, and let $S\in\La(\ell_{2})$. If $V(BS-SA)U^{-1}\in\La(\ell_{2},\ell_{2-\epsilon})$, then
\begin{align*}
\norm{\abs{B}S-S\abs{A}}_{\La(\ell_{2})}\leq C\|V(BS-SA)U^{-1}\|_{\La(\ell_{2},\ell_{2-\epsilon})}
\end{align*}
and if $V(BS-SA)U^{-1}\in\La(\ell_{2+\epsilon},\ell_{2})$ then
\begin{align*}
\norm{\abs{B}S-S\abs{A}}_{\La(\ell_{2})}\leq C\|V(BS-SA)U^{-1}\|_{\La(\ell_{2+\epsilon},\ell_{2})}.
\end{align*}
In particular, if $V(B-A)U^{-1}\in\La(\ell_{2},\ell_{2-\epsilon})$, then
\begin{align*}
\norm{\abs{B}-\abs{A}}_{\La(\ell_{2})}\leq C\|V(B-A)U^{-1}\|_{\La(\ell_{2},\ell_{2-\epsilon})}
\end{align*}
and if $V(B-A)U^{-1}\in\La(\ell_{2+\epsilon},\ell_{2})$, then
\begin{align*}
\norm{\abs{B}-\abs{A}}_{\La(\ell_{2})}\leq C\|V(B-A)U^{-1}\|_{\La(\ell_{2+\epsilon},\ell_{2})}.
\end{align*}
\end{corollary}

\section{Lipschitz estimates on the ideal of $p$-summing operators}\label{p-summing operators}

Let $H$ be a separable infinite-dimensional Hilbert space. It was shown in \cite{A-1982} that a matrix $M=\{m_{jk}\}_{j,k=1}^\infty$ is a Schur multiplier on the Hilbert-Schmidt class $\mathcal{S}_{2}\subset\La(H)$ if and only if $\sup_{j,k}|m_{jk}|<\infty$. 
By \cite{Pelcz}, $\mathcal{S}_{2}$ coincides with the Banach ideal $\Pi_{p}(H)$ of all $p$-summing operators (see the definition below) for all $1\leq p<\infty$. Hence a matrix $M=\{m_{jk}\}_{j,k=1}^\infty$ is a Schur multiplier on $\Pi_{p}(H)$ if and only if $\sup_{j,k}|m_{jk}|<\infty$.
In Lemma \ref{any bounded seq} we show that the same statement is true for the Banach ideal $\Pi_{p}(\ell_{p^*},\ell_{p})$ in $\La(\ell_{p^{*}},\ell_{p})$, for all $1<p<\infty$. As a corollary we obtain operator Lipschitz estimates on $\Pi_{p}(\ell_{p^{*}},\ell_{p})$ for each Lipschitz function $f$ on $\C$.

Let $X$ and $Y$ be Banach spaces and $1\leq p< \infty$. An operator
$S:X\to Y$ is \emph{$p$-absolutely summing} if there exists a
constant $C$ such that for each $n\in\N$ and each collection
$\{x_{j}\}_{j=1}^{n}\subseteq X$,
\begin{equation}\label{def_p_sum}
\Big(\sum_{j=1}^{n}\norm{S(x_{j})}_{Y}^{p}\Big)^{\frac{1}{p}}\leq C
\sup\limits_{\|x^{*}\|_{X^{*}}\leq 1} \Big(\sum_{j=1}^{n}\abs{\langle x^{*},
x_{j}\rangle}^{p}\Big)^{\frac{1}{p}}.
\end{equation}
The smallest such constant is denoted by $\pi_{p}$, and $\Pi_{p}(X,Y)$ is the space of $p$-absolutely summing operators from $X$ to $Y$. We let $\Pi_{p}(X):=\Pi_{p}(X,X)$. By Propositions 2.3, 2.4 and 2.6 in \cite{Diestel_absPsum}, $(\Pi_{p}(X,Y),\pi_{p}(\cdot))$ is a Banach ideal in $\La(X,Y)$.

Below we consider $p$-absolutely summing operators from $\ell_{p^{*}}$ to $\ell_{p}$. We first present the following result.

\begin{lemma}\label{lem_p_sum_criteria}
Let $1<p<\infty$ and $S=\{s_{jk}\}_{j,k=1}^{\infty}$. Then $S\in \Pi_{p}(\ell_{p^*},\ell_{p})$ if and only if
\begin{align*}
c_{p}:=\Big(\sum_{j=1}^\infty\sum_{k=1}^\infty|s_{jk}|^{p}\Big)^{\frac{1}{p}}<\infty.
\end{align*}
In this case, $\pi_{p}(S)= c_{p}$.
\end{lemma}
\begin{proof}
In \cite[Example 2.11]{Diestel_absPsum} it is shown that, if $c_{p}<\infty$, then $S\in \Pi_{p}(\ell_{p^*},\ell_{p})$ with $\pi_{p}(S)\leq c_{p}$. For the converse, let $n\in\N$ and let $x_{j}:=e_{j}\in \ell_{p^{*}}$ for $1\leq j\leq n$. By \eqref{def_p_sum} (with $X=\ell_{p^{*}}$ and $Y=\ell_{p}$),
\begin{align*}
\Big(\sum_{k=1}^{n}\sum_{j=1}^{\infty}\abs{s_{jk}}^{p}\Big)^{\frac{1}{p}}\leq \pi_{p}(S).
\end{align*}
Letting $n$ tend to infinity concludes the proof.
\end{proof}

For the following corollary of Lemma \ref{lem_p_sum_criteria}, recall that a matrix $M$ is said to be a Schur multiplier on a subspace $\Ic\subseteq \La(\ell_{p},\ell_{q})$ if $S\mapsto M\ast S$ is a bounded map on $\Ic$. Recall also the definition of the standard triangular truncation $\Tr_{\Delta}$ from \eqref{classical_tr_tr}.

\begin{corollary}\label{any bounded seq} 
Let $p\in(1,\infty)$ and let $M=\{m_{jk}\}_{j,k=1}^\infty$ be a matrix. Then $M$ is a Schur multiplier on $\Pi_{p}(\ell_{p^{*}},\ell_{p})$ if and only if $\sup_{j,k}\,\abs{m_{jk}}<\infty$. In this case, 
\begin{align*}
\norm{M}_{\La(\Pi_{p}(\ell_{p^{*}},\ell_{p}))}=\sup_{j,k}\,\abs{m_{jk}}.
\end{align*}
In particular, $\Tr_{\Delta}\in\La(\Pi_{p}(\ell_{p^{*}},\ell_{p}))$ with
$\norm{\Tr_{\Delta}}_{\La(\Pi_{p}(\ell_{p^{*}},\ell_{p}))}=1$.
\end{corollary}

Observe that $\Tr_{\Delta}\notin\La(\La(\ell_{p^{*}},\ell_{p}))$ for $p^{*}\geq p$, by Proposition \ref{Pil-Ben} $(ii)$. Nevertheless, $\Tr_{\Delta}$ is bounded on the ideal $\Pi_{p}(\ell_{p^{*}},\ell_{p})\subset \La(\ell_{p^{*}},\ell_{p})$ for all $p\in(1,\infty)$.

For a Lipschitz function $f:\mathbb C\to \C$, write
\begin{align*}
\norm{f}_{\mathrm{Lip}}:=
\sup_{\substack{z_{1},z_{2}\in\mathbb C\\ z_{1}\neq
z_{2}}}\frac{\abs{f(z_1)-f(z_2)}}{\abs{z_1-z_2}}.
\end{align*}
Moreover, let $\ph_{f}:\C^{2}\to \C$ be given by
\begin{align*}
\ph_{f}(\lambda_{1},\lambda_{2}):= \left\{\begin{array}{ll}
\frac{\abs{\lambda_1}-\abs{\lambda_2}}{\lambda_1-\lambda_2}&\textrm{if }\lambda_{1}\neq\lambda_{2}\\
0&\textrm{otherwise}
\end{array}.\right.
\end{align*} 
\vanish{
\begin{align*}
\norm{f}_{\mathrm{Lip}}:=
\sup_{\substack{z_{1},z_{2}\in\mathbb C\\ z_{1}\neq
z_{2}}}\frac{\abs{f(z_1)-f(z_2)}}{\abs{z_1-z_2}}
\end{align*}
if $f$ is Lipschitz. For sequences $\lambda=\{\lambda_{j}\}_{j=1}^{\infty}$ and
$\mu=\{\mu_k\}_{k=1}^{\infty}$ of complex numbers, consider the matrix
\begin{align*}
M_{\ph_{f}}^{\lambda,\mu}:=\{\ph_f(\lambda_{j},\mu_{k})\}_{j,k=1}^{\infty}.
\end{align*}
The following result is an immediate consequence of Lemma \ref{any
bounded seq}.

\begin{proposition}\label{for phi_f} 
Let $p\in(1,\infty)$, $f:\C\to\C$ and $C\geq 0$. Then the following statements are equivalent:
\begin{itemize}
\item For all complex sequences $\lambda$ and $\mu$, $M_{\ph_{f}}^{\lambda,\mu}$ is a Schur multiplier on $\Pi_{p}(\ell_{p^{*}},\ell_{p})$ with $\norm{M_{\ph_{f}}^{\lambda,\mu}}_{\La(\Pi_{p}(\ell_{p^{*}},\ell_{p}))}\leq C$;
\item $f$ is Lipschitz with $\norm{f}_{\mathrm{Lip}}\leq C$.
\end{itemize}
\end{proposition}
}

We now prove our main result concerning commutator estimates on $\Pi_{p}(\ell_{p^{*}},\ell_{p})$. 

\begin{theorem}\label{th_p_summing_main}
Let $p\in(1,\infty)$, $A\in\La_{\ud}(\ell_{p^{*}})$ and $B\in\La_{\ud}(\ell_{p})$. Let $f:\C\to \C$ be Lipschitz. Then
\begin{align}\label{final est_p_sum}
\pi_{p}(f(B)S-Sf(A))\leq \diagA\diagB \norm{f}_{\mathrm{Lip}}\pi_{p}(BS-SA)
\end{align}
for all $S\in \La(\ell_{p^{*}},\ell_{p})$ such that $BS-SA\in\Pi_{p}(\ell_{p^{*}},\ell_{p})$.
\end{theorem}
\begin{proof}
Let $\lambda=\{\lambda_{j}\}_{j=1}^{\infty}$ and $\mu=\{\mu_{k}\}_{k=1}^{\infty}$ be sequences such that $A\in\La_{\ud}(\ell_{p^{*}}, \lambda, U)$ and $B\in\La_{\ud}(\ell_{p}, \mu, V)$ for certain $U\in\La(\ell_{p^{*}})$ and $V\in\La(\ell_{p})$. 

It follows directly from \eqref{def_p_sum} that, if $\{S_{m}\}_{m=1}^{\infty}\subseteq \Pi_{p}(\ell_{p^{*}},\ell_{p})$ is a $\pi_{p}$-bounded sequence which SOT-converges to some $S\in \La(X,Y)$, then $S\in \Pi_{p}(\ell_{p^{*}},\ell_{p})$ with $\pi_{p}(S)\leq \limsup_{m\to\infty}\pi_{p}(S_{m})$. Hence, by Remark \ref{other ideals}, it suffices to prove that $\sup_{n\in\N}\|T_{\ph_{f},n}^{\lambda,\mu}\|_{\La(\Pi_{p}(\ell_{p^{*}},\ell_{p}))}\leq \norm{f}_{\mathrm{Lip}}$, where 
\begin{align*}
T_{\ph_{f},n}^{\lambda,\mu}(S)=\sum_{j,k=1}^{n}\ph_{f}(\lambda_{j},\mu_{k})\Pp_{k}S\Pp_{j}\qquad(S\in \Pi_{p}(\ell_{p^{*}},\ell_{p}))
\end{align*}
for $n\in\N$. Fix $n\in\N$ and note that $T_{\ph_{f},n}^{\lambda,\mu}(S)=M\ast S$ for $S\in\Pi_{p}(\ell_{p^{*}},\ell_{p})$, where $M=\{m_{jk}\}_{j,k=1}^{\infty}$ is the matrix given by $m_{jk}=\ph_{f}(\lambda_{j},\mu_{k})$ for $1\leq j,k\leq n$, and $m_{jk}=0$ otherwise. Then 
\begin{align*}
\sup_{j,k}\,\abs{m_{jk}}\leq \sup_{j,k}\,\abs{\ph_{f}(\lambda_{j},\mu_{k})}\leq \norm{f}_{\mathrm{Lip}}.
\end{align*}
Corollary \ref{any bounded seq} now concludes the proof.
\end{proof}

\vanish{
\begin{problem}
As it was mentioned above,
$\mathcal S_2=\Pi_p(H)$ for every $1\le p< \infty.$
The result in Lemma \ref{any bounded seq} works for 
$\mathcal S_2$ and for $\Pi_{p}(l_{p^*},\ell_{p}),$ $1<p<\infty.$
Therefore, it is reasonable to expect that the ideal 
$\Pi_{q}(l_{p^*},\ell_{p})$  satisfies 
Lemma \ref{any bounded seq}
for $1<p,q<\infty$.
\end{problem}
}

\begin{problem}
Let $p,q,r\in[1,\infty]$ be such that $q\geq 2$ and $\frac{1}{q}-\frac{1}{p}<\frac{1}{2}$ and $\frac{1}{r}=\frac{1}{p}-\frac{1}{q}+\frac{1}{2}$. Then the Schatten class $\mathcal{S}_{r}$ coincides with the Banach ideal $\Pi_{p,q}(\ell_{2})$ of $(p,q)$-summing operators on $\ell_{2}$ 
(for the definition of $(p,q)$-summing operators see \cite{Diestel_absPsum}).
Hence, by \cite{DDPS}, the standard triangular truncation is bounded on $\Pi_{p,q}(\ell_{2})$ for these indices. For which $r_{1},r_{2}\in[1,\infty]$ is the standard triangular truncation bounded on $\Pi_{p,q}(\ell_{r_{1}},\ell_{r_{2}})$? Are there other non-trivial ideals $\Ic$ in $\La(\ell_{p},\ell_{q})$ such that the standard triangular truncation is bounded on
$\Ic$? As shown in Theorem \ref{th_p_summing_main} (and the rest of this paper), answers to these questions are linked to commutator estimates for diagonalizable operators.
\end{problem}

\section{Matrix estimates}\label{matrix estimates}

In this section we apply the theory developed in Sections \ref{double operator integrals and Lipschitz estimates}, \ref{unconditional basis spaces} and \ref{lp-spaces} to finite-dimensional spaces, and in doing so we obtain Lipschitz estimates which are independent of the dimension of the underlying space. The results from Section \ref{p-summing operators}, in particular Theorem \ref{th_p_summing_main}, also yield dimension-independent estimates on finite-dimensional spaces, but we leave the straightforward derivation of these results to the reader.

\subsection{Finite-dimensional spaces}\label{finite-dimensional spaces}

Let $n\in\N$ and let $X$ be an $n$-dimensional Banach space with basis $\left\{e_{1},\ldots,e_{n}\right\}\subset X$. For $1\leq k\leq n$, let $\Pp_{k}\in\La(X)$ be the projection given by $\Pp_{k}(x_{1}e_{1}+\ldots+x_{n}e_{n}):=x_{k}e_{k}$ for $(x_{1},\ldots,x_{n})\in\C^{n}$. Recall that an operator $A\in\La(X)$ is diagonalizable if there exists $U\in\La(X)$ invertible such that
\begin{align}\label{diagonal}
UAU^{-1}=\sum_{k=1}^{n}\lambda_{k}\Pp_{k}
\end{align}
for some $(\lambda_{1},\ldots,\lambda_{n})\in\C^{n}$. In this case we write $A\in\La_{\ud}(X,\{\lambda_{j}\}_{j=1}^{n},U)$. Recall also the definition of spectral and scalar type operators from Section \ref{scalar type operators}. 

\begin{lemma}\label{diagonalizable}
Each $A\in\La(X)$ is a spectral operator. Furthermore, $A$ is a scalar type operator if and only if $A$ is diagonalizable. If $A\in\La_{\ud}(X,\{\lambda_{j}\}_{j=1}^{n},U)$ then the spectral measure $E$ of $A$ is given by $E(U)=0$ if $U\cap{\rm sp}(A)=\emptyset$, and $E(\left\{\lambda\right\})=\sum_{\lambda_{j}=\lambda}U^{-1}\Pp_{j}U$ for $\lambda\in{\rm sp}(A)$.
\end{lemma}
\begin{proof}
It was already remarked in Section \ref{unconditional basis spaces} that any diagonalizable operator is a scalar type operator, with spectral measure as specified. By~\cite[Theorem XV.4.5]{Dunford-Schwartz1971}, an operator $T\in\La(Y)$ on an arbitrary Banach space $Y$ is a spectral operator if and only if $T=S+N$ for a scalar type operator $S\in\La(Y)$ and a generalized nilpotent operator $N\in\La(Y)$ such that $NS=SN$, and this decomposition is unique. Combining this with the Jordan decomposition for matrices yields that each $A\in\La(X)$ is a spectral operator. If $A$ is a scalar type operator, then the Jordan decomposition yields a commuting diagonalizable $S$ and a nilpotent $N$ such that $A=S+N$. By the uniqueness of such a decomposition~\cite[Theorem XV.4.5]{Dunford-Schwartz1971}, $N=0$ and $A=S$ is diagonalizable.
\end{proof}

Let $A\in\La_{\ud}(X,\{\lambda_{j}\}_{j=1}^{n},U)$. As in \eqref{functional calculus matrices_infinite},
\begin{align*}
f(A)=U^{-1}\Big(\sum_{k=1}^{n}f(\lambda_{k})\Pp_{k}\Big)U.
\end{align*}
Let $Y$ be a finite-dimensional Banach space. A norm $\norm{\cdot}$ on $\La(X,Y)$ is \emph{symmetric} if
\begin{itemize}
\item $\norm{RST}\leq \norm{R}_{\La(Y)}\norm{S}\norm{T}_{\La(X)}$ for all $R\in \La(Y)$, $S\in\La(X,Y)$ and $T\in \La(X)$;
\item $\norm{x^{*}\!\otimes y}=\norm{x^{*}}_{X^{*}}\norm{y}_{Y}$ for all $x^{*}\in X^{*}$ and $y\in Y$.
\end{itemize}
Clearly $(\La(X,Y),\norm{\cdot})$ is a Banach ideal in the sense of Section \ref{spaces of operators} if and only if $\norm{\cdot}$ is symmetric.

The following result extends inequalities which were known for self-adjoint operators on finite-dimensional Hilbert spaces and unitarily invariant norms~(see e.g.~\cite{Kosaki1992} and \cite[Chapter X]{Bhatia97}), to diagonalizable operators on finite-dimensional Banach spaces and symmetric norms. Note that, for general finite-dimensional Banach spaces $X$ and $Y$, a symmetric norm on $\La(X,Y)$ is not unitarily invariant. Let $\A:=\A(\C\times\C)$ be as in Section \ref{algebras of functions}, and for $f\in\B(\C)$ let $\ph_{f}(\lambda_{1},\lambda_{2}):=\frac{f(\lambda_{2})-f(\lambda_{1})}{\lambda_{2}-\lambda_{1}}$ for $(\lambda_{1},\lambda_{2})\in\C^{2}$ with $\lambda_{1}\neq\lambda_{2}$, as in \eqref{divided difference}. Recall the definition of the spectral constants $\funA$ and $\funB$ from Section \ref{scalar type operators}. The following is a direct corollary of Theorem \ref{general perturbation inequality}, since $\La(X,Y)$ has the strong convex compactness property.

\begin{theorem}\label{matrix inequality}
Let $f\in\B(\C)$ be such that $\ph_{f}$ extends to an element of $\A$. Let $X$ and $Y$ be finite-dimensional Banach spaces, $\norm{\cdot}$ a symmetric norm on $\La(X,Y)$, and let $A\in\La_{\ud}(X)$ and $B\in\La_{\ud}(Y)$. Then
\begin{align*}
\norm{f(B)S-Sf(A)}\leq 16\,\funA\funB\big\|\ph_{f}\big\|_{\A}\norm{BS-SA}
\end{align*}
for all $S\in\La(X,Y)$. In particular, if $X=Y$,
\begin{align}\label{matrix perturbation}
\norm{f(B)-f(A)}\leq 16\,\funA\funB\big\|\ph_{f}\big\|_{\A}\norm{B-A}.
\end{align}
\end{theorem}

\begin{corollary}\label{besov matrices}
There exists a universal constant $C\geq 0$ such that the following holds. Let $X$ and $Y$ be finite-dimensional Banach spaces and $\norm{\cdot}$ a symmetric norm on $\La(X,Y)$. Let $f\in\Br^{1}_{\infty,1}(\R)$, and let $A\in\La_{\ud}(X)$ and $B\in\La_{\ud}(Y)$ be such that ${\rm sp}(A)\cup{\rm sp}(B)\subseteq \R$. Then
\begin{align*}
\norm{f(B)S-Sf(A)}\leq C\,\funA\funB\norm{f}_{\Br^{1}_{\infty,1}(\R)}\!\norm{BS-SA}
\end{align*}
for all $S\in \La(X,Y)$. In particular, if $X=Y$,
\begin{align*}
\norm{f(B)-f(A)}\leq C\,\funA\funB\norm{f}_{\Br^{1}_{\infty,1}(\R)}\!\norm{B-A}.
\end{align*}
\end{corollary}

\begin{remark}\label{dimension-independent}
Let $\sigma_1,\sigma_2\subset \C$ be finite sets. Then any $\ph:\sigma_1\times \sigma_2\rightarrow \C$ belongs to $\A(\sigma_1\times \sigma_2)$. Indeed, one can find a representation as in \eqref{function representation} by letting $\Omega$ be finite and solving a system of linear equations. Therefore Theorem \ref{general perturbation inequality} yields an estimate
\begin{align*}
\norm{f(B)S-Sf(A)}\leq 16\,\funA\funB\big\|\ph_{f}\big\|_{\A({\rm sp}(A)\times{\rm sp}(B))}\norm{BS-SA}
\end{align*}
as in \eqref{main inequality} for all $f\in\B(\C)$. This might lead one to think that the restriction in Theorem \ref{matrix inequality} that $\ph_{f}$ extends to an element of $\A$ is not really necessary. However, for general $f\in\B(\C)$ the norm $\big\|\ph_{f}\big\|_{\A({\rm sp}(A)\times{\rm sp}(B))}$ may blow up as the number of points in ${\rm sp}(A)$ and ${\rm sp}(B)$ grows to infinity. Indeed, as remarked before, letting $f\in\B(\C)$ be the absolute value function and considering the operator norm, a dimension-independent estimate as in \eqref{matrix perturbation} does not hold for all self-adjoint operators on all finite-dimensional Hilbert spaces \cite[(X.25)]{Bhatia97}. Hence $\ph_{f}$ does not extend to an element of $\A$, and one cannot expect to obtain Theorem \ref{matrix inequality} or Corollary \ref{besov matrices} for all bounded Borel functions on $\C$.
\end{remark}

\subsection{The absolute value function}

We now apply the results for the absolute value function from previous sections to finite-dimensional spaces. Throughout this section, let $f(t):=\abs{t}$ for $t\in\R$.

First note that Lemma \ref{write as schur multiplier} and Proposition \ref{truncation estimate_inf} relate commutator estimates to estimates for triangular truncation operators for general symmetric norms on matrix spaces. 

For $n\in\N$ and $p\in[1,\infty)$ let $\ell_{p}^{n}$ denote $\C^{n}$ with the $p$-norm
\begin{align*}
\norm{(x_{1},\ldots,x_{n})}_{p}:=\Big(\sum_{j=1}^{n}\abs{x_{j}}^{p}\Big)^{1/p}\qquad\big((x_{1},\ldots, x_{n})\in\C^{n}\big),
\end{align*}
and let $\ell_{\infty}^{n}$ be $\C^{n}$ with the norm 
\begin{align*}
\norm{(x_{1},\ldots,x_{n})}_{\infty}:=\max_{1\leq j\leq n}\,\abs{x_{j}}\qquad\big((x_{1},\ldots, x_{n})\in\C^{n}\big).
\end{align*}
Theorem \ref{truncation estimate_inf_lplq} immediately yields the following result. It shows that, although the Lipschitz estimate
\begin{align*}
\norm{f(B)-f(A)}_{\La(\ell_{p}^{n},\ell_{q}^{n})}\leq \const \norm{B-A}_{\La(\ell_{p}^{n},\ell_{q}^{n})}
\end{align*}
does not hold with a constant independent of the dimension $n$ for $f$ the absolute value function, $p=q=2$ and all self-adjoint operators on $\ell_{2}^{n}$, one can nevertheless obtain such estimates for $p<q$, $p=q=1$ or $p=q=\infty$ by considering diagonalizable operators $A$ and $B$ for which $\diagA, \diagB\leq M$, for some fixed $M\geq 1$. For $A$ a diagonalizable operator, recall the definition of $\diagA$ from \eqref{definition diagonalizable constant}.

\begin{theorem}\label{p<q matrices}
Let $p,q\in[1,\infty]$ with $p<q$, $p=q=1$ or $p=q=\infty$. Then there exists a constant $C\geq 0$ such that the following holds. Let $n\in\N$ and let $A\in\La_{\ud}(\ell_{p}^{n})$ and $B\in\La_{\ud}(\ell_{q}^{n})$ have real spectrum. Then 
\begin{align*}
\norm{f(B)S-Sf(A)}_{\La(\ell_{p}^{n},\ell_{q}^{n})}\leq C\diagA\diagB\norm{BS-SA}_{\La(\ell_{p}^{n},\ell_{q}^{n})}
\end{align*}
for all $S\in\La(\ell_{p}^{n},\ell_{q}^{n})$. In particular,
\begin{align*}
\norm{f(B)-f(A)}_{\La(\ell_{p}^{n},\ell_{q}^{n})}\leq C\diagA\diagB\norm{B-A}_{\La(\ell_{p}^{n},\ell_{q}^{n})}.
\end{align*}
\end{theorem}

Of course Corollaries \ref{to l2} and \ref{from l2} imply that for $p=2$ or $q=2$ and $A$ or $B$ self-adjoint, the estimates in Theorem \ref{p<q matrices} simplify.

For $p\geq q$, Propositions \ref{case p geq q1} and \ref{case p geq q2} yield dimension-independent estimates. We state the estimates which follow from Proposition \ref{case p geq q1}, the analogous result that follows from Proposition \ref{case p geq q2} should be obvious. 

\begin{proposition}\label{pgeqq matrices}
Let $p,q\in(1,\infty]$ with $p\geq q$. For each $s<q$ there exists a constant $C\geq 0$ such that the following holds. Let $n\in\N$, and let $A\in\La_{\ud}(\ell_{p}^{n},\lambda,U)$ and $B\in\La_{\ud}(\ell_{q}^{n},\mu,V)$ have real spectrum. Then
\begin{align*}
\|f(B)S-Sf(A)\|_{\La(\ell_{p}^{n},\ell_{q}^{n})}\leq C \|U\|_{\La(\ell_{p}^{n})}\|V^{-1}\|_{\La(\ell_{q}^{n})}\|V(BS-SA)U\|_{\La(\ell_{p}^{n},\ell_{s}^{n})}
\end{align*}
for all $S\in\La(\ell_{p}^{n},\ell_{q}^{n})$. In particular,
\begin{align*}
\|f(B)-f(A)\|_{\La(\ell_{p}^{n},\ell_{q}^{n})}\leq C \|U\|_{\La(\ell_{p}^{n})}\|V^{-1}\|_{\La(\ell_{q}^{n})}\|V(B-A)U\|_{\La(\ell_{p}^{n},\ell_{s}^{n})}.
\end{align*}
\end{proposition}

In the case $p=q=2$, Corollary \ref{case p=q=2} implies the following. Again we only consider $\epsilon\in(0,1]$, for simplicity, but the result extends in an obvious manner to other $\epsilon>0$. We write $f(A)=\abs{A}=\sqrt{A^{*}A}$ for a normal operator $A$ on $\ell_{2}^{n}$.

\begin{corollary}\label{case p=q=2 matrices}
For each $\epsilon\in(0,1]$ there exists a constant $C\geq 0$ such that the following holds. Let $n\in\N$ and let $A\in\La_{\ud}(\ell_{2}^{n},\lambda,U)$ and $B\in\La_{\ud}(\ell_{2}^{n},\mu,V)$ be self-adjoint operators, with $U$ and $V$ unitaries. Then
\begin{align}\label{l2 estimate matrices}
\|\abs{B}-\abs{A}\|_{\La(\ell_{2}^{n})}\leq C\min(\|V(B-A)U^{-1}\|_{\La(\ell_{2}^{n},\ell_{2-\epsilon}^{n})},\|V(B-A)U^{-1}\|_{\La(\ell_{2+\epsilon}^{n},\ell_{2}^{n})}).
\end{align}
\end{corollary}

We note that \eqref{l2 estimate matrices} in turn yields, for instance, the following estimate:
\begin{align*}
\|\abs{B}-\abs{A}\|_{\La(\ell_{2}^{n})}\leq C\|B-A\|_{\La(\ell_{2}^{n})}\min(\|U^{-1}\|_{\La(\ell_{2}^{n},\ell_{2-\epsilon}^{n})},\|V\|_{\La(\ell_{2+\epsilon}^{n},\ell_{2}^{n})}).
\end{align*}

\subsubsection{Acknowledgements}

Most of the work on this article was performed during a research
visit of the first-named author to the University of New South
Wales. He would like to thank UNSW and its faculty, in particular
Michael Cowling, for their wonderful hospitality. We are indebted
to many colleagues in Delft and at UNSW for their suggestions and
ideas. In particular, we would like to mention Ian Doust, Ben de
Pagter, Denis Potapov and Dmitriy Zanin. We also thank Albrecht
Pietsch for useful comments and references concerning 
$p$-summing operators.

\bibliographystyle{plain}
\bibliography{Bibliografie}

\end{document}